\documentclass[12pt,a4paper]{article}
\usepackage{graphicx}
\usepackage{amsfonts}
\usepackage{amsmath}
\usepackage{amssymb}
\usepackage{amsthm}
\usepackage[ansinew]{inputenc}
\usepackage{latexsym}
\usepackage{tabularx}
\usepackage{color}
\usepackage{enumerate}
\usepackage[active]{srcltx}
\allowdisplaybreaks
\newtheorem{thm}{Theorem}

\newtheorem{lem}{Lemma}
\newtheorem{prop}{Proposition}
\newtheorem{defn}{Definition}
\newtheorem{rem}{Remark}

\newcommand{\norm}[1]{\left\Vert#1\right\Vert}
\newcommand{\abs}[1]{\left\vert#1\right\vert}

\newcommand{\To}{\rightarrow}

\newcommand{\A}{\mathcal{A}}

\newcommand{\calH}{\mathcal{H}}

\newcommand{\bsa}{\boldsymbol{a}}

\newcommand{\bslambda}{\boldsymbol{\lambda}}

\newcommand{\bsk}{\boldsymbol{k}}
\newcommand{\bsl}{\boldsymbol{l}}
\newcommand{\bsv}{\boldsymbol{v}}

\newcommand{\bsx}{\boldsymbol{x}}

\newcommand{\bsh}{\boldsymbol{h}}

\newcommand{\bsr}{\boldsymbol{r}}

\newcommand{\bsb}{\boldsymbol{b}}

\newcommand{\cO}{{\cal O}}

\newcommand{\bsy}{\boldsymbol{y}}

\newcommand{\cH}{{\cal H}}

\newcommand{\bszero}{\boldsymbol{0}}
\newcommand{\bsone}{\boldsymbol{1}}
\newcommand{\rd}{\,\mathrm{d}}
\newcommand{\e}{\,\varepsilon}

\newcommand{\NN}{\mathbb{N}}
\newcommand{\ZZ}{\mathbb{Z}}
\newcommand{\RR}{\mathbb{R}}
\newcommand{\LL}{\mathbb{L}}

\newcommand{\cG}{{\cal G}}

\newcommand{\uu}{\mathfrak{u}}

\newcommand{\lstd}{\Lambda^{\rm{std}}}
\newcommand{\lall}{\Lambda^{\rm{all}}}
\newcommand{\infp}{\operatornamewithlimits{inf\phantom{p}}}
\newcommand{\ee}{{\rm e}}
\newcommand{\APP}{{\rm APP}}

\makeatletter
\newcommand{\rdots}{\mathinner{\mkern1mu\lower-1\p@\vbox{\kern7\p@\hbox{.}}
\mkern2mu \raise4\p@\hbox{.}\mkern2mu\raise7\p@\hbox{.}\mkern1mu}}
\makeatother

\date{}

\begin{document}

\title{Approximation in Hermite spaces of smooth functions}

\author{Christian Irrgeher\thanks{C. Irrgeher is supported by the Austrian Science Fund
(FWF): Project F5509-N26, which is a part of the Special Research Program
"Quasi-Monte Carlo Methods: Theory and Applications".}\;,
Peter Kritzer\thanks{P. Kritzer is supported by the Austrian Science Fund (FWF):
Project F5506-N26, which is a part of the Special Research Program "Quasi-Monte Carlo Methods:
Theory and Applications".}\;, Friedrich Pillichshammer\thanks{F. Pillichshammer is
partially supported by the Austrian Science Fund (FWF): Project F5509-N26,
which is a part of the Special Research Program "Quasi-Monte Carlo Methods:
Theory and Applications".},\\ Henryk Wo\'{z}niakowski\thanks{H. Wo\'zniakowski
is supported by the NSF and by the National Science Centre, Poland, based on the decision DEC-2013/09/B/ST1/04275.}}

\maketitle

\begin{abstract}
We consider $\mathbb{L}_2$-approximation of elements of a Hermite space
of analytic functions over $\mathbb{R}^s$. The Hermite space is a weighted reproducing
kernel Hilbert space of real valued functions for which the Hermite
coefficients decay exponentially fast. The weights are defined in
terms of two sequences $\boldsymbol{a} = \{a_j\}$ and $\boldsymbol{b} = \{b_j\}$ of
positive real numbers. We study the $n$th minimal worst-case error
$e(n,{\rm APP}_s; \lstd)$ of all algorithms that use $n$ information
evaluations from the class $\Lambda^{{\rm std}}$ which only allows
function evaluations to be used.

We study (uniform) exponential convergence of the $n$th minimal
worst-case error, which means that $e(n,{\rm APP}_s; \Lambda^{{\rm std}})$ converges
to zero exponentially fast with increasing $n$. Furthermore, we
consider how the error depends on the dimension~$s$. To this end,
we study the minimal number of information evaluations needed to
compute an $\varepsilon$-approximation by considering several
notions of tractability which are defined with respect to~$s$ and
$\log \varepsilon^{-1}$. We derive necessary and sufficient
conditions on the sequences $\boldsymbol{a}$ and $\boldsymbol{b}$ for obtaining
exponential error convergence, and also for obtaining the various
notions of tractability. It turns out that the conditions on the
weight sequences are almost the same as for the information class $\Lambda^{{\rm all}}$
which uses all linear functionals. The results are also
constructive as the considered algorithms are based on tensor
products of Gauss-Hermite rules for multivariate integration. The
obtained results are compared with the analogous results for
integration in the same Hermite space. This allows us to give a new 
sufficient condition for EC-weak tractability for integration.
\end{abstract}



\section{Introduction}\label{sectractability}

In this paper we study $\LL_2$-approximation of functions
belonging to a certain reproducing kernel Hilbert space
$\mathcal{H}(K_s)$ of $s$-variate functions defined on $\RR^s$
with reproducing kernel $K_s:\RR^s \times \RR^s \rightarrow \RR$.
We are interested in approximating the embedding operators
$\APP_s: \mathcal{H}(K_s)\rightarrow \LL_2(\RR^s,\varphi_s)$ with
$$\APP_s(f)=f,$$ where $\varphi_s$ denotes the density of the
$s$-dimensional standard Gaussian measure.

We consider the worst-case setting. In this case it
follows from general results on information-based complexity, see, e.g.,
\cite{TWW88} or \cite[Section~4]{NW08}, that linear algorithms are
optimal. So we approximate $\APP_s$ by a linear algorithm
$A_{n,s}$ using $n$ information evaluations either from the class
$\Lambda^{\rm{std}}$ of standard information which consists of
only function evaluations or from the class $\Lambda^{\rm{all}}$
of all continuous linear functionals. That is,
\begin{align*}
A_{n,s}(f)=\sum_{j=1}^n \alpha_j L_j(f)\qquad \mbox{for all}\qquad
f\in \mathcal{H}(K_s),
\end{align*}
where $L_j$ belongs to the dual space of $\mathcal{H}(K_s)$ , i.e., 
$L_j\in \mathcal{H}(K_s)^*$, for the class $\Lambda^{{\rm
all}}$, whereas $L_j(f)=f(\bsx_j)$ for all $f\in \mathcal{H}_s$, with 
$\bsx_j\in \RR^s$ for the class $\Lambda^{{\rm std}}$, and
$\alpha_j\in \LL_2(\RR^s,\varphi_s)$ for all $j=1,2,\dots,n$. Since
$\mathcal{H}(K_s)$ is a reproducing kernel Hilbert space we
obviously have $\lstd\subseteq\lall$. In this paper we will mainly
concentrate on the approximation problem with respect to the class
$\lstd$, because the problem for the class $\lall$ is covered by
\cite{IKPW15}.

We measure the error of an algorithm $A_{n,s}$ in terms of the
\textit{worst-case error}, which is defined as
\begin{align}\label{eq:wce}
e^{{\rm app}}(\mathcal{H}(K_s),A_{n,s}):=\sup_{\substack{f \in
\mathcal{H}(K_s) \\ \|f\|_{K_s} \le 1}}
\norm{\APP_s(f)-A_{n,s}(f)}_{\LL_2},
\end{align}
where $\norm{\cdot}_{K_s}$ denotes the norm in $\mathcal{H}(K_s)$,
and $\norm{\cdot}_{\LL_2}$ denotes the norm in $\LL_2(\RR,\varphi_s)$
which is given by $$\|g\|_{\LL_2}=\left(\int_{\RR^s} |g(\bsx)|^2
\varphi_s(\bsx) \rd \bsx\right)^{1/2}\ \ \ \mbox{ for }\ g \in
\LL_2(\RR,\varphi_s).$$ The \textit{$n$th minimal (worst-case)
error} is given by
\begin{align}\label{eq:nthmwce}
e(n,\APP_s;\Lambda):=\infp_{A_{n,s}} e^{{\rm
app}}(\mathcal{H}(K_s), A_{n,s}),
\end{align}
where the infimum is taken over all admissible algorithms
$A_{n,s}$ using information from the class
$\Lambda\in\{\lall,\lstd\}$.

For $n=0$, we consider algorithms that do not use any information
evaluation, and therefore we use $A_{0,s}\equiv 0$. The error  of
$A_{0,s}$ is called the \textit{initial (worst-case) error} and is
given by
\begin{align}\label{eq:initialerror}
e(0,\APP_s):=\sup_{\substack{f \in \mathcal{H}(K_s) \\
\|f\|_{K_s} \le 1}} \norm{\APP_s(f)}_{\LL_2}=\norm{\APP_s}.
\end{align}

When studying algorithms $A_{n,s}$, we do not only want to control
how their errors depend on $n$, but also how they depend on the
dimension $s$. This is of particular importance for
high-dimensional problems. To this end, we define, for
$\e\in(0,1)$ and $s\in\NN$, the {\it information complexity} by
\begin{align*}
n(\varepsilon,\APP_s;\Lambda):=
\min\left\{n\,:\,e(n,\APP_s;\Lambda)\le\varepsilon \right\}
\end{align*}
as the minimal number of information evaluations needed to obtain
an $\e$-approximation to $\APP_s$. In this case, we speak of the
\textit{absolute error criterion}. Alternatively, we can also
define the information complexity as
\begin{align*}
n(\varepsilon,\APP_s;\Lambda):=
\min\left\{n\,:\,e(n,\APP_s;\Lambda)\le\varepsilon\,
e(0,\APP_s)\right\},
\end{align*} 
i.e., as the minimal number of information evaluations needed to
reduce the initial error by a factor of $\e$. In this case we
speak of the \textit{normalized error criterion}.

The specific problem considered in this paper has the convenient
property that the initial error is 
one, and the absolute and normalized error criteria coincide.

\subsection{Exponential convergence and tractability}

Since the particular weighted function space we are going to define in
Section~\ref{secspace} is such that its elements are infinitely many times
differentiable and even analytic, it is natural to expect that the $n$th 
minimal error converges to zero very quickly as $n$ increases. Indeed, 
we would like to achieve exponential convergence of the $n$th minimal 
errors, and we first define this type of convergence in detail.\\

By exponential convergence we mean that there exist functions
$q:\NN=\{1,2,\ldots\}\to(0,1)$ and $p,C:\NN\to (0,\infty)$ such
that
\begin{align*}
e(n,\APP_s;\Lambda)\le C(s)\, q(s)^{\,n^{\,p(s)}}\qquad \mbox{for
all}\quad s, n\in \NN.
\end{align*}
Obviously, the functions $q(\cdot)$ and $p(\cdot)$ are not
uniquely defined. For instance, we can take an arbitrary number
$q\in(0,1)$, define the function $C_1$ as
\begin{align*}
C_1(s)=\left(\frac{\log\,q}{\log\,q(s)}\right)^{1/p(s)},
\end{align*}
and then
\begin{align*}
C(s)\, q(s)^{\,n^{\,p(s)}}=C(s)\,q^{\,(n/C_1(s))^{p(s)}}.
\end{align*}
We prefer to work with the latter bound which was also considered
in~\cite{DKPW13,IKLP14,KPW14}.

\begin{defn}\rm
We say that we achieve  \emph{exponential convergence} (EXP) if
there exist a number $q\in(0,1)$ and functions $p,C,C_1:\NN\to
(0,\infty)$ such that
\begin{align}\label{exrate}
e(n,\APP_s;\Lambda)\le C(s)\, q^{\,(n/C_1(s))^{\,p(s)}}\qquad
\mbox{for all}\quad s, n\in \NN.
\end{align}
If \eqref{exrate} holds, then the largest possible rate of
exponential convergence is defined as
\begin{align*}
p^*(s)= \sup \{\,p\in (0,\infty): & \ \exists \, C, C_1:\NN\rightarrow
(0, \infty) \mbox{ such that }\\
& \ \forall n \in \NN: e(n,\APP_s;\Lambda) \le C(s)
q^{(n/C_1(s))^p} \}.
\end{align*}
\end{defn}

\begin{defn}\rm
We say that we achieve \emph{uniform exponential convergence}
(UEXP) if the function $p$ in \eqref{exrate} can be taken as a
constant function, i.e., $p(s)=p>0$ for all $s\in\NN$.
Furthermore, let
\begin{align*}
p^* = \sup \{ p \in (0, \infty): \exists \, C, C_1: & \ \NN
\rightarrow (0, \infty) \mbox{ such that }\\
& \ \forall n, s \in \NN: e(n,\APP_s;\Lambda)
\le C(s) q^{(n/C_1(s))^p} \}
\end{align*}
denote the largest rate of uniform exponential convergence.
\end{defn}

We note, see \cite{DKPW13, DLPW11}, that if \eqref{exrate} holds
and $e(0,\APP_s)=1$ then 
\begin{align}\label{exrate2}
n(\e,\APP_s;\Lambda) \le \left\lceil C_1(s) \left(\frac{\log C(s)
+ \log \e^{-1}}{\log q^{-1}}\right)^{1/p(s)}\right\rceil \quad
\mbox{for all}\  s\in \NN\ \mbox{and}\ \e\in (0,1).
\end{align}
Conversely, if~\eqref{exrate2} holds then
\begin{align*}
e(n+1,\APP_s;\Lambda)\le C(s)\, q^{\,(n/C_1(s))^{\,p(s)}}\qquad
\mbox{for all}\quad s,n\in \NN.
\end{align*}
This means that \eqref{exrate} and \eqref{exrate2} are practically
equivalent. Note that $1/p(s)$ determines the power of $\log
\e^{-1}$ in the information complexity, whereas $\log q^{-1}$ only
affects the multiplier of $\log^{1/p(s)}\e^{-1}$. From this point
of view, $p(s)$ is more important than $q$.

From \eqref{exrate2} we learn that exponential convergence implies
that asymptotically, with respect to $\e$ tending to zero, we need
$\mathcal{O}(\log^{1/p(s)} \e^{-1})$ information evaluations 
to obtain an $\e$ approximation.
However, it is not
clear how long we have to wait to see this nice asymptotic
behavior especially for large $s$. This, of course, depends on how
$C(s),C_1(s)$ and $p(s)$ depend on $s$, and this is the subject of
tractability. The following tractability notions were already
considered in \cite{DKPW13,DLPW11,IKLP14,IKPW15, KPW14,KPW14a}. 
The nomenclature was introduced in \cite{KPW14a}. In this paper 
we define $\log 0=0$ for convention.

\begin{defn}\rm
We say that we have:
\begin{enumerate}[(a)]
\item \emph{Exponential Convergence-Weak Tractability {\rm (EC-WT)}} if
\begin{align*}
\lim_{s+\e^{-1}\to\infty}\frac{\log n(\varepsilon,\APP_s;\Lambda)}{s+\log\e^{-1}}=0.
\end{align*}
\item \emph{Exponential
Convergence-Polynomial Tractability {\rm (EC-PT)}} if there exist
non-negative numbers $c,\tau_1,\tau_2$ such that
\begin{align*}
n(\varepsilon,\APP_s;\Lambda)\le
c\,s^{\,\tau_1}\,(1+\log\e^{-1})^{\,\tau_2}\qquad \mbox{for
all}\quad s\in\NN, \e\in(0,1).
\end{align*}
\item \emph{Exponential Convergence-Strong Polynomial Tractability
{\rm (EC-SPT)}} if there exist non-negative numbers $c$ and $\tau$
such that
\begin{align*}
n(\varepsilon,\APP_s;\Lambda)\le
c\,(1+\log\,\e^{-1})^{\,\tau}\qquad \mbox{for all}\quad s\in\NN,
\e\in(0,1).
\end{align*}
The exponent $\tau^*$ of {\rm EC-SPT} is defined as the infimum of
$\tau$ for which the above relation holds.
\end{enumerate}
\end{defn}

EC-WT means that we rule out the cases for which $n(\e,\APP_s;\Lambda)$ 
depends exponentially on $s$ and $\log \e^{-1}$. 
EC-PT means that the information complexity depends at most polynomially 
on $s$ and $\log \e^{-1}$ whereas EC-SPT means that $n(\e,\APP_s;\Lambda)$ 
is bounded at most polynomially in $\log\e^{-1}$, independently of $s$.

We remark that in many papers tractability has been studied for
problems where we do not have exponential but usually polynomial
error convergence. For this kind of problems, tractability has
been defined by studying how the information complexity depends on
$s$ and $\e^{-1}$, for a detailed survey of such results we refer
to~\cite{NW08,NW10,NW12}. With the notions of EC-tractability
considered in~\cite{DKPW13, DLPW11, IKLP14, KPW14, KPW14a} and in
the present paper, however, we study how the information
complexity depends on $s$ and $\log \e^{-1}$. We remark that $\log
\e^{-1}$ also corresponds to the number of bits of desired
accuracy, cf.\ \cite{PP14}.

We collect some well-known relations:
\begin{prop} \label{propo1}
We have:
\begin{enumerate}[(i)]
\item\label{ec_lem_1} {\rm EC-SPT} $\Rightarrow$ {\rm EC-PT}
$\Rightarrow$ {\rm EC-WT}. 
\item\label{ec_lem_2}
{\rm EC-PT} (and therefore also {\rm EC-SPT}) implies {\rm UEXP}.
\item\label{ec_lem_3} {\rm EC-WT} implies that
$e(n,\APP_s;\Lambda)$ converges to zero faster than any power of
$n^{-1}$ as $n$ goes to infinity, i.e.,
$$\lim_{n \rightarrow \infty} n^{\alpha} e(n,\APP_s;\Lambda)=0
\quad \mbox{ for all }\ \alpha \in \RR^+ \ \mbox{and all}\ s \in \NN.$$
\item\label{ec_lem_4} If we have  {\rm UEXP},
$e(n,\APP_s;\Lambda)\le C(s)\, q^{(n/C_1(s))^{p}}$, then:
\begin{itemize}
\item $C(s)=\exp(\exp(o(s)))$ and $C_1(s)=\exp(o(s))$
$\Rightarrow$ {\rm EC-WT}; \item $C(s)=\exp(\cO(s^{\vartheta}))$ and
$C_1(s) =\cO(s^{\eta})$ for some $\vartheta,\eta>0$ $\Rightarrow$ {\rm EC-PT}; \item
$C(s)=\cO(1)$ and $C_1(s)=\cO(1)$  $\Rightarrow$ {\rm EC-SPT}.
\end{itemize}
\end{enumerate}
\end{prop}

\begin{proof}
{\it (\ref{ec_lem_1})} is clear.  A proof of {\it (\ref{ec_lem_2})} can be
found in \cite{DKPW13,KPW14a} and {\it (\ref{ec_lem_3})} and
{\it (\ref{ec_lem_4})} are shown in \cite{KPW14a}.
\end{proof}

Of course Point {\it (\ref{ec_lem_2})} of Proposition \ref{propo1} is the motivation for the
use of the prefix EC (exponential convergence) in our notation.

The goal of this paper is to find relations between the concepts EXP, UEXP,
and the various tractability notions, as well as necessary and
sufficient conditions on the weights of the considered function
space for which these concepts hold, mostly for the class $\lstd$.

\subsection{Hermite spaces with infinite smoothness}\label{secspace}

We briefly summarize some facts on \textit{Hermite polynomials};
for further details, we refer to~\cite{IL} and the references
therein. For $k \in \NN_0=\{0,1,2,\ldots\}$ the $k$th Hermite
polynomial is given by
\begin{align*}
H_k(x)=\frac{(-1)^k}{\sqrt{k!}} \exp(x^2/2) \frac{\rd^k}{\rd x^k}
\exp(-x^2/2),
\end{align*}
which is sometimes also called normalized probabilistic Hermite
polynomial. Here we follow the definition given in \cite{B98}, but
we remark that there are slightly different ways to introduce
Hermite polynomials, see, e.g., \cite{szeg}. For $s \ge 2$,
$\bsk=(k_1,\ldots,k_s)\in \NN_0^s$ and $\bsx=(x_1,\ldots,x_s)\in
\RR^s$, we define $s$-dimensional Hermite polynomials by
\begin{align*}
H_{\bsk}(\bsx)=\prod_{j=1}^s H_{k_j}(x_j).
\end{align*}
It is well known, see again~\cite{B98}, that the sequence of
Hermite polynomials $\{H_{\bsk}\}_{\bsk \in \NN_0^s}$ forms an
orthonormal basis of the function space $\LL_2(\RR^s,\varphi_s)$,
where $\varphi_s$ denotes the density of the $s$-dimensional
standard Gaussian measure,
\begin{align*}
\varphi_s(\bsx)=\frac{1}{(2 \pi)^{s/2}} \exp\left(-\frac{1}{2}\,
\bsx \cdot \bsx\right)\quad\textnormal{for all }\bsx\in\RR^s,
\end{align*}
where ``$\cdot$'' is the standard Euclidean inner product in
$\RR^s$. We write $\varphi:=\varphi_1$.

Similarly to what has been done in~\cite{IL}, we are now going to
define function spaces based on Hermite polynomials. These spaces
are Hilbert spaces with a \textit{reproducing kernel}. For details
on reproducing kernel Hilbert spaces, we refer to~\cite{A50}.

Let $r: \NN_0^s \rightarrow \RR^+$ be a summable function, i.e.,
$\sum_{\bsk\in\NN_0^s} r(\bsk) < \infty$. Define a kernel function
\begin{align*}
K_{r}(\bsx,\bsy)=\sum_{\bsk \in \NN_0^s} r(\bsk) H_{\bsk}(\bsx)
H_{\bsk}(\bsy)\ \ \ \ \mbox{ for }\ \ \bsx,\bsy \in \RR^s
\end{align*}
and an inner product
\begin{align*}
\langle f,g\rangle_{K_{r}} =\sum_{\bsk \in \NN_0^s}
\frac{1}{r(\bsk)} \widehat{f}(\bsk) \widehat{g}(\bsk),
\end{align*}
where 
\begin{align*}
\widehat{f}(\bsk)=\int_{\RR^s} f(\bsx) H_{\bsk}(\bsx)
\varphi_s(\bsx)\rd \bsx
\end{align*}
is the $\bsk$th \textit{Hermite
coefficient} of $f$. Note that $K_r(\bsx,\bsy)$ is well defined
for all $\bsx,\bsy \in \RR^s$, since
\begin{align*}
|K_r(\bsx,\bsy)| \le \sum_{\bsk}  r(\bsk) |H_{\bsk}(\bsx)| \,
|H_{\bsk}(\bsy)| \le \frac{1}{\sqrt{\varphi_s(\bsx)
\varphi_s(\bsy)}}\sum_{\bsk}  r(\bsk)< \infty,
\end{align*}
since Cramer's bound for Hermite polynomials, see, e.g., \cite[p.
324]{sansone}, states that
\begin{align*}
\vert H_{k}(x)\vert\leq
\frac{1}{\sqrt{\varphi(x)}}\quad\textnormal{ for all } k\in\NN_0.
\end{align*}

Let $\mathcal{H}(K_r)$ be the reproducing kernel Hilbert space
corresponding to $K_{r}$, which we will call a Hermite space.
The norm in $\mathcal{H}(K_r)$ is given by $\| f
\|_{K_{r}}^2 =\langle f , f \rangle_{K_{r}}$. From this we see
that the functions in $\mathcal{H}(K_r)$ are characterized by the
decay rate of their Hermite coefficients, which is regulated by
the function $r$. Roughly speaking, the faster $r$ decreases as
$\bsk$ grows, the faster the Hermite coefficients of the elements
of $\mathcal{H}(K_r)$ decrease. In \cite{IL}, the case of
polynomially decreasing $r$ as well as exponentially decreasing
$r$ was considered. In \cite{IKLP14} further results were obtained 
for numerical integration for exponentially decreasing $r$, and 
in \cite{IKPW15} for approximation using information from 
$\Lambda^{\rm all}$. In this paper, we continue the work on
exponentially decreasing $r$ for approximation using information from 
$\Lambda^{\rm std}$, thereby extending the results of
\cite{IKLP14,IKPW15,IL}.

To define our function $r$, we first introduce two weight
sequences of positive real numbers, $\bsa=\{a_j\}$ and
$\bsb=\{b_j\}$ such that
\begin{align*}
0 < a_1\leq a_2\leq \cdots\ \ \ \mbox{ and }\ \ \ b_{\ast}:=\inf_j
b_j\ge 1.
\end{align*}

Furthermore, we fix a parameter $\omega\in (0,1)$. For a vector
$\bsk=(k_1,\ldots,k_s)\in\NN_0^s$, we consider
\begin{align*}
r(\bsk)=\omega_{\bsk}:=\omega^{\sum_{j=1}^s a_j k_j^{b_j}}.
\end{align*}
For simplicity we assume without loss of generality that $a_1\geq 1$, because we can 
always modify $\omega$ in such a way that $a_1$ is greater than or equal 
to $1$. 

We modify the notation for the kernel function to
\begin{align*}
K_r(\bsx,\bsy)=K_{s,\bsa,\bsb,\omega} (\bsx,\bsy)=\sum_{\bsk\in\NN_0^s}
\omega_{\bsk} H_{\bsk}(\bsx) H_{\bsk} (\bsy).
\end{align*}
From now on, we deal with the corresponding reproducing kernel
Hilbert space $\calH(K_{s,\bsa,\bsb,\omega})$. Our concrete choice
of $r$ now decreases exponentially fast as $\bsk$ grows, which
influences the smoothness of the elements in
$\cH(K_{s,\bsa,\bsb,\omega})$. Indeed, if $b_{\ast} \ge 1$ it can
be shown that functions $f \in \cH(K_{s,\bsa,\bsb,\omega})$ are
analytic, see \cite{IKLP14}. More precisely, we
have that for all $\bsx\in \RR^s$ the Taylor expansion of $f$
centered at $\bsx$  converges in a ball with radius
$\rho(\omega)>0$ around $\bsx$. It can also be shown that this radius $\rho(\omega)$ is
independent of $\bsx$ and $\lim_{\omega \rightarrow 0}
\rho(\omega)=\infty$ and $\lim_{\omega \rightarrow 1}
\rho(\omega)=0$.

\begin{rem}\rm
Apparently the assumption $b_{\ast}\ge 1$ has technical reasons,
see, e.g., the footnote on page~\pageref{bjgr1}, and pages~\pageref{bjgr1(2)} 
and \pageref{bjgr1(3)}. However, the assumption $b_{\ast}\ge 1$ is also 
essential in showing that functions $f \in \cH(K_{s,\bsa,\bsb,\omega})$ are
analytic, see \cite{IKLP14}. For the moment it must remain an open question 
whether our results are also correct if $b_{\ast} \in (0,1)$. However, in the
case $b_{\ast} \in (0,1)$, we can show that functions $f \in \cH(K_{s,\bsa,\bsb,\omega})$ 
belong to the Gevrey class of index $1/b_*$, which is work in progress.
\end{rem}

We remark that reproducing kernel Hilbert spaces of a similar
flavor were previously studied
in~\cite{DKPW13,DLPW11,KPW14,KPW14a}, but the functions considered
there were one-periodic functions defined on the unit cube
$[0,1]^s$. Here, we study functions which are defined on the
$\RR^s$, which is a major difference. Obviously,
$\cH(K_{s,\bsa,\bsb,\omega})$ contains all polynomials on the
$\RR^s$, but there are further functions of practical interest
which belong to such spaces.  For example, it is easy to verify,
see again~\cite{IKLP14}, that $f(\bsx)=\exp(\bslambda \cdot \bsx)$
is an element of the Hilbert space $\cH(K_{s,\bsa,\bsone,\omega})$
for any weight sequence $\bsa$ and any
$\bslambda \in \RR^s$. Functions of a similar form occur in
problems of financial derivative pricing, see, e.g., \cite{ll}.

Multivariate integration in $\cH(K_{s,\bsa,\bsb,\omega})$ has been
studied in \cite{IKLP14} and will be discussed further in Section~\ref{secint} of this paper.

\section{$\LL_2$-approximation in $\calH(K_{s,\bsa,\bsb,\omega})$}

Let $\APP_s: \calH (K_{s,\bsa,\bsb,\omega})\To \LL_2
(\RR^s,\varphi_s)$ with $\APP_s(f)=f$. In order to approximate
$\APP_s$ in the norm $\| \cdot \|_{\LL_2}$  we use linear algorithms
$A_{n,s}$, which use $n$ information evaluations and which are of
the form
\begin{align*}
A_{n,s}(f)=\sum_{k=1}^n \alpha_k L_k (f)\quad \mbox{for}\
f\in\calH (K_{s,\bsa,\bsb,\omega}),
\end{align*}
where each $\alpha_k$ is a function from $\LL_2 (\RR^s,\varphi_s)$
and each $L_k$ is a continuous linear functional defined on $\calH
(K_{s,\bsa,\bsb,\omega})$ from a permissible class
$\Lambda\in\{\lall,\lstd\}$ of information.

The worst-case error $e^{\mathrm{app}}$ of an algorithm $A_{n,s}$
is defined as in \eqref{eq:wce}
and the $n$th minimal worst-case error for the information class
$\Lambda$ is given by \eqref{eq:nthmwce}.
The initial error, defined by~\eqref{eq:initialerror}, is
\begin{align*}
e(0,\APP_s)=\norm{\APP_s}= \sup_{\substack{f\in \calH
(K_{s,\bsa,\bsb,\omega})\\ \norm{f}_{K_{s,\bsa,\bsb,\omega}}\le
1}} \norm{f}_{\LL_2}= \sup_{\substack{f\in \calH
(K_{s,\bsa,\bsb,\omega})\\ \norm{f}_{K_{s,\bsa,\bsb,\omega}}\le
1}} \norm{f}_{K_{s,\bsa,\bsb,\omega}}=1,
\end{align*}
since we always have $\norm{f}_{K_{s,\bsa,\bsb,\omega}} \ge
\norm{f}_{\LL_2}$ and equality is obtained for the constant function~
$1$ which certainly belongs to $\calH (K_{s,\bsa,\bsb,\omega})$.
This means that the approximation problem is well normalized and
that the absolute and the normalized error criteria coincide,
i.e., the {\it information complexity} is
\begin{align*}
  n(\varepsilon,\APP_s;\Lambda):= \min\left\{n\,:\,e(n,\APP_s;\Lambda)\le\varepsilon \right\}.
\end{align*}

\subsection{Results for $\LL_2$-approximation for the class $\Lambda^{{\rm all}}$}

$\LL_2$-approximation for the class $\lall$ defined over very general
Hilbert spaces with exponential weights is discussed in \cite{IKPW15}.
Since the Hermite space $\mathcal{H}(K_{s,\bsa,\bsb,\omega})$ with 
weight sequences $\bsa$ and $\bsb$ fits into the setting of \cite{IKPW15}, we know
that the following results hold for the class $\lall$:
\begin{enumerate}
    \item EXP holds for arbitrary $\bsa$ and $\bsb$ and $p^{*}(s)=1/B(s)$ with $B(s):=\sum_{j=1}^s b_j^{-1}$.
    \item UEXP holds iff $\bsa$ is an arbitrary sequence and $\bsb$ such that $B:=\sum_{j=1}^\infty\frac1{b_j}<\infty$. If so then $p^*=1/B$.
    \item We have
\begin{align*}
\textnormal{EC-WT}\qquad &\Leftrightarrow\qquad \lim_{j\to\infty}a_j=\infty,\\
\textnormal{EC-PT}\qquad &\Leftrightarrow\qquad
B:=\sum_{j=1}^\infty\frac1{b_j}<\infty\quad\textnormal{and}\quad \alpha^*:=\liminf_{j\to\infty}\frac{\log\,a_j}j>0,\\
\textnormal{EC-SPT}\qquad &\Leftrightarrow\qquad
B:=\sum_{j=1}^\infty\frac1{b_j}<\infty\quad\textnormal{and}\quad
\alpha^*:=\liminf_{j\to\infty}\frac{\log\,a_j}j>0.
\end{align*}
Then the exponent $\tau^*$ of {\rm EC-SPT} satisfies
$\max\left(B,\frac{\log\,2}{\alpha^*}\right)\le \tau^*\le
B+\frac{\log\,2}{\alpha^*}$. In particular, we have $\textnormal{EC-PT} \Leftrightarrow \textnormal{EC-SPT}$.
\item The following notions are
equivalent:
\begin{align*}
\textnormal{EC-PT}&, \ \textnormal{EC-PT+EXP}, \ \textnormal{EC-PT+UEXP},\\
&\textnormal{EC-SPT}, \
\textnormal{EC-SPT+EXP}, \ \textnormal{EC-SPT+UEXP}.
\end{align*}
\end{enumerate}

\subsection{Results for $\LL_2$-approximation for the class $\Lambda^{{\rm std}}$}

We present the main results of this paper in the following theorem:

\begin{thm}\label{th:result}
Consider $\LL_2$-approximation defined over the Hermite space
$\mathcal{H}(K_{s,\bsa,\bsb,\omega})$ with weight sequences $\bsa$
and $\bsb$ satisfying $0 < a_1 \le a_2 \le a_3 \le \ldots$ and
$\inf_j b_j \ge 1$. The following results hold for the class
$\Lambda^{\rm{std}}$.
\begin{enumerate}
\item\label{exp} {\rm EXP} holds for arbitrary $\bsa$ and $\bsb$
and
\begin{align*}
p^{*}(s)=\frac{1}{B(s)} \qquad \textnormal{with}\qquad
B(s):=\sum_{j=1}^s\frac1{b_j}.
\end{align*}

\item\label{uexp} {\rm UEXP} holds iff $\bsa$ is an arbitrary
sequence and $\bsb$ is such that
\begin{align*}
B:=\sum_{j=1}^\infty\frac1{b_j}<\infty.
\end{align*}
If this is the case then $p^*=1/B$.

\item\label{tract} We have
\begin{enumerate}
    \item[a.]\label{wt} {\rm EC-WT} iff $\lim_{j\to\infty}a_j=\infty$,
    \item[b.]\label{pt} {\rm EC-PT} iff  {\rm EC-SPT},
    \item[c.]\label{spt} {\rm EC-SPT} iff $B:=
    \sum_{j=1}^\infty\frac1{b_j}<\infty$ and $\alpha^*:=\liminf_{j\to\infty}\frac{\log\,a_j}j>0$.\\[0.2cm]
Then the exponent $\tau^*$ of {\rm EC-SPT} satisfies
\begin{align*}
\max\left(B,\frac{\log\,2}{\alpha^*}\right)\le \tau^*\le
B+\frac{\log\,3}{\alpha^*}.
\end{align*}
In particular, $\alpha^*=\infty$ implies $\tau^*=B$.
\end{enumerate}

\end{enumerate}
\end{thm}

The results we achieve for the information class $\lstd$ match
those for the class $\lall$, although the upper bound on the exponent 
of EC-SPT is slightly different. From Theorem \ref{th:result} we 
see once more that EC-PT implies UEXP, cf.\ Proposition~\ref{propo1}.

We cannot determine the exponent of EC-SPT exactly but we get an
upper and a lower bound such that we know
$\tau^*\in[\max\left(B,(\log 2)/\alpha^*\right),B+(\log 3)/\alpha^*]$.\\

The proof of Theorem~\ref{th:result} will be given in Section~\ref{prThm1}. First we collect some auxiliary results
in the following section.

\subsection{Auxiliary results for the proof of Theorem \ref{th:result}}\label{applambdastd}

\subsubsection{Gauss-Hermite rules}

A one-dimensional {\it Gauss-Hermite rule of order $n$} is a
linear integration rule $Q_{n}$ of the form
\begin{align*}
Q_n(f)=\sum_{i=1}^n \alpha_i f(x_i)
\end{align*}
that is exact for all polynomials $p$ of degree less than $2n$,
\begin{align*}
\int_{\RR} p(x) \varphi(x) \rd x=\sum_{i=1}^n \alpha_i p(x_i).
\end{align*}
The nodes
$x_1,\ldots,x_n \in \RR$ are the zeros of the $n$th
Hermite polynomial $H_n$ and the integration weights $\alpha_i$ are given by
\begin{align*}
\alpha_i=\frac{1}{n H_{n-1}^2(x_i)}\qquad \textnormal{ for
$i=1,2,\ldots,n$},
\end{align*}
see \cite{hildebrand}. We stress that the weights $\alpha_i$ are
all positive. The following lemma summarizes a few basic facts on
Gauss-Hermite rules.
\begin{lem}\label{leGH}
Let $n \in \NN$. Then we have:
\begin{enumerate}
\item $Q_{n}(H_0)=\sum_{i=1}^n \alpha_i=1$; \item for $k \in
\{1,2,\ldots,2n-1\}$ we have $Q_n(H_k)=0$; \item for $k \in \{2n,
2n+1,\ldots \}$ we have
\begin{align*}
|Q_n(H_k)| \leq\begin{cases}\sqrt[4]{8 \pi} & \mbox{ if $k$ is
even},\\0 & \mbox{ if $k$ is odd}. \end{cases}
\end{align*}
\end{enumerate}
\end{lem}
\begin{proof}
See \cite[Proof of Proposition~1]{IKLP14}.
\end{proof}

For integration in the multivariate case, we use the tensor
product of one-dimensional Gauss-Hermite rules. Let
$m_1,\ldots,m_s \in \NN$ and let $n=m_1 m_2\cdots m_s$. For
$j=1,2,\ldots,s$ let $$Q_{m_j}^{(j)}(f)=\sum_{i=1}^{m_j}
\alpha_i^{(j)} f(x_i^{(j)})$$ be one-dimensional Gauss-Hermite
rules of order $m_j$ with nodes $x_1^{(j)},\ldots,x_{m_j}^{(j)}$
and with weights $\alpha_1^{(j)},\ldots,\alpha_{m_j}^{(j)}$,
respectively. Then we apply the $s$-dimensional tensor product
rule $$Q_{n,s}=Q_{m_1}^{(1)} \otimes \cdots \otimes
Q_{m_s}^{(s)},$$ i.e.,
\begin{equation}\label{cpghr}
Q_{n,s}(f)=\sum_{i_1=1}^{m_1}\ldots \sum_{i_s=1}^{m_s} \alpha_{i_1}^{(1)}
\cdots \alpha_{i_s}^{(s)} f(x_{i_1}^{(1)},\ldots,x_{i_s}^{(s)}).
\end{equation}
By $\cG_{n,s}^{\perp}$ we denote the set
\begin{align*}
\cG_{n,s}^{\perp}=\{\bsv \in \NN_0^s \, : \, \textnormal{ for all
} j=1,\ldots,s \textnormal{ either } v_j=0, \textnormal{ or }
v_j\geq 2m_j \textnormal{ and } v_j \textnormal{ even} \}.
\end{align*}

We will make use of the following result.
\begin{lem}\label{grile}
Let $Q_{n,s}$ be as in \eqref{cpghr}. For any $g$ of the form
\begin{align*}
g(\bsx)=\sum_{\bsv\in\NN_0^s}\widehat{g}(\bsv) H_{\bsv}(\bsx) \ \
\mbox{ for all }\ \bsx\in\RR^s,
\end{align*}
we have
\begin{align*}
\left|\int_{\RR^s} g(\bsx)\varphi_s(\bsx)\rd \bsx -
Q_{n,s}(g)\right| \le \sum_{\bsv \in \cG_{n,s}^{\perp} \setminus
\{\bszero\}}| \widehat{g}(\bsv)|  (\sqrt[4]{8\pi})^{\abs{\bsv}_*},
\end{align*}
where we put $\abs{\bsv}_*:=\abs{\{j: v_j\neq 0\}}$ for
$\bsv=(v_1,\ldots,v_s)\in \NN_0^s$.
\end{lem}
\begin{proof}
Using the results from Lemma~\ref{leGH} as well as the 
orthonormality of the Hermite polynomials we have
\begin{align*}
\abs{\int_{\RR^s} g(\bsx)\varphi_s(\bsx)\rd \bsx - Q_{n,s}(g)} = & \abs{\sum_{\bsv \in \NN_0^s \setminus \{\bszero\}} \widehat{g}(\bsv) Q_{n,s}(H_{\bsv})} \\
\le & \sum_{\bsv \in \NN_0^s \setminus \{\bszero\}} \abs{\widehat{g}(\bsv)} \prod_{j=1}^{s}|Q_{m_j}^{(j)} (H_{v_j})|\\
\le & \sum_{\bsv \in \cG_{n,s}^{\perp} \setminus \{\bszero\}}
\abs{\widehat{g}(\bsv)} (\sqrt[4]{8\pi})^{\abs{\bsv}_*},
\end{align*}
as desired.
\end{proof}

\subsubsection{Error analysis in $\cH(K_{s,\bsa,\bsb,\omega})$}

We proceed in a similar way as in \cite{DKPW13,KSW06}. Let
$M>1$, and define
\begin{align}\label{eqAsMweighted}
\A (s,M):=\left\{\bsh\in\NN_0^s: \omega_{\bsh}^{-1}<M\right\}.
\end{align}
For $f \in \cH(K_{s,\bsa,\bsb,\omega})$ and $\bsh \in \NN_0^s$
define
\begin{align*}
f_{\bsh}(\bsx):=f(\bsx) H_{\bsh}(\bsx) \ \ \mbox{ for }\
\bsx\in\RR^s.
\end{align*}
We approximate $f \in
\mathcal{H}(K_{s,\bsa,\bsb,\omega})$ by algorithms of the form
\begin{align}\label{eqdefAnsM}
A_{n,s,M}(f)(\bsx)= \sum_{\bsh \in \mathcal{A}(s,M)}
Q_{n,s}(f_{\bsh})  H_{\bsh}(\bsx)\qquad\textnormal{for }
\bsx\in\RR^s,
\end{align}
where $Q_{n,s}$ is a Gauss-Hermite rule of the form
\eqref{cpghr}. The choice of $M$ will be given below. Then we have
\begin{align*}
(f-A_{n,s,M}(f))(\bsx)=& \sum_{\bsh \not \in  \mathcal{A}(s,M)}
\widehat{f}(\bsh) H_{\bsh}(\bsx) +\sum_{\bsh \in
\mathcal{A}(s,M)}(\widehat{f}(\bsh)-Q_{n,s}(f_{\bsh}))
H_{\bsh}(\bsx).
\end{align*}
Using Parseval's identity we obtain
\begin{align}\label{approx_err}
\|f-A_{n,s,M}(f)\|_{\LL_2}^2& = \sum_{\bsh \not \in  \mathcal{A}(s,M)} |\widehat{f}(\bsh)|^2 +
\sum_{\bsh \in \mathcal{A}(s,M)} |\widehat{f}(\bsh)- Q_{n,s}(f_{\bsh})|^2\nonumber\\
&= \sum_{\bsh \not \in  \mathcal{A}(s,M)} |\widehat{f}(\bsh)|^2 +
\sum_{\bsh \in
\mathcal{A}(s,M)}\left|\int_{\RR^s}f_{\bsh}(\bsx)\varphi_s(\bsx)\rd\bsx-Q_{n,s}(f_{\bsh})\right|^2.
\end{align}
We have
\begin{align}\label{approx_err1}
\sum_{\bsh \not \in  \mathcal{A}(s,M)}|\widehat{f}(\bsh)|^2 =
\sum_{\bsh \not \in\mathcal{A}(s,M)}|\widehat{f}(\bsh)|^2
\omega_{\bsh}\,\omega_{\bsh}^{-1} \le\frac{1}{M}
\|f\|_{K_{s,\bsa,\bsb,\omega}}^2.
\end{align}

Now we estimate the second term in \eqref{approx_err}.
Unfortunately, in general, $f \in  \cH(K_{s,\bsa,\bsb,\omega})$
does not imply $f_{\bsh} \in  \cH(K_{s,\bsa,\bsb,\omega})$. However, we
can show the following result.

\begin{lem}\label{le_Hermitseriesfh}
The function $f_{\bsh}$ can be pointwise represented as a Hermite
series
\begin{align*}
f_{\bsh}(\bsx)=\sum_{\bsk\in\NN_0^s}\widehat{f_{\bsh}}(\bsk)H_{\bsk}(\bsx)\qquad\textnormal{for all }\bsx\in\RR^s.
\end{align*}
\end{lem}

For technical reasons, we defer the proof of this lemma to the end of this subsection.
With the help of Lemma~\ref{le_Hermitseriesfh} we can estimate the
integration error of $Q_{n,s}$ for functions of the form
$f_{\bsh}$.

\begin{lem}\label{le5}
For $f$ in the unit ball of $\cH(K_{s,\bsa,\bsb,\omega})$ and
$\bsh \in \mathcal{A}(s,M)$ we have
\begin{align*}
&\lefteqn{\left|\int_{\RR^s} f_{\bsh}(\bsx)\varphi_s(\bsx)\rd\bsx-Q_{n,s}(f_{\bsh}) \right|^2}\\
&\quad\le 2^s \left(\prod_{j=1}^{\min (s,j(x_M))} \left\lceil
\left(\frac{\log M}{a_j
\log\omega^{-1}}\right)^{1/b_j}\right\rceil\right)M^K
\sum_{\bsv\in \cG_{n,s}^{\bot} \setminus
\{\bszero\}}(\sqrt{8\pi})^{\abs{\bsv}_*}\omega^{\frac{1}{2}\sum_{j=1}^s
a_j (v_j/2)^{b_j}},
\end{align*}
where
\begin{align}
x_M&:=\frac{\log M}{\log \omega^{-1}},\nonumber\\
j(x)&:=\sup\{j\in\NN: x>a_j\},\nonumber\\
K&=K(\omega):= 3k-1+\frac{2\log (1+\omega^k)}{\log \omega^{-1}},
\ \ \ \mbox{with $k:=\max\left(1,\left\lceil
\frac{\log(\omega^{-1/8}-1)}{\log\omega}\right\rceil\right)$.}\label{def_K}
\end{align}
\end{lem}

\begin{proof}
According to Lemma~\ref{le_Hermitseriesfh} we can apply Lemma
\ref{grile} to the second term in \eqref{approx_err} and obtain
\begin{align}\label{eqintfh}
\left|\int_{\RR^s} f_{\bsh}(\bsx)\varphi_s(\bsx)\rd\bsx-
Q_{n,s}(f_{\bsh}) \right|^2 \le\left(\sum_{\bsv\in
\cG_{n,s}^{\bot} \setminus \{\bszero\}} |\widehat{f_{\bsh}}(\bsv)|
(\sqrt[4]{8\pi})^{\abs{\bsv}_*} \right)^2.
\end{align}
For fixed $\bsh\in\A (s,M)$ and $\bsv\in \cG_{n,s}^{\bot}
\setminus \{\bszero\}$, we have
\begin{align*}
\widehat{f_{\bsh}}(\bsv)&=\int_{\RR^s} f(\bsx) H_{\bsh} (\bsx) H_{\bsv}(\bsx)\varphi_s (\bsx)\rd\bsx\\
&=\int_{\RR^s} f(\bsx) \left(\prod_{j=1}^s H_{h_j} (x_j)
H_{v_j}(x_j)\right)\varphi_s (\bsx)\rd\bsx.
\end{align*}
Now we write the product of two Hermite polynomials as a linear combination of Hermite polynomials. To this end
we write $t_j=\min (v_j,h_j)$ and
$T_j=\max (v_j,h_j)$ for $j\in\{1,\ldots,s\}$. With this notation
we have, using a result from \cite[p.~1]{C62},
\begin{align*}
H_{h_j}(x_j) H_{v_j}(x_j)=\sum_{r_j=0}^{t_j}\left(\frac{t_j !}{T_j
!}\right)^{1/2} {T_j \choose t_j -r_j}
\frac{((\abs{h_j-v_j}+2r_j)!)^{1/2}}{r_j !}
H_{\abs{h_j-v_j}+2r_j}(x_j).
\end{align*}
Hence
\begin{align*}
&\lefteqn{\widehat{f}_{\bsh}(\bsv)=}\\
&=\int_{\RR^s} f(\bsx)
\left(\prod_{j=1}^s \sum_{r_j=0}^{t_j}\left(\frac{t_j !}{T_j !}\right)^{1/2}
{T_j \choose t_j -r_j} \frac{((\abs{h_j-v_j}+2r_j)!)^{1/2}}{r_j !} H_{\abs{h_j-v_j}+2r_j}(x_j)\right)\varphi_s (\bsx)\rd\bsx\\
&=\int_{\RR^s} f(\bsx) \sum_{r_1=0}^{t_1}\cdots\sum_{r_s=0}^{t_s}
\Bigg(\prod_{j=1}^s \left(\frac{t_j !}{T_j !}\right)^{1/2} {T_j \choose t_j -r_j} \frac{((\abs{h_j-v_j}+2r_j)!)^{1/2}}{r_j !}\\
&\qquad\hspace{4cm}\times H_{\abs{h_j-v_j}+2r_j}(x_j)\Bigg)\varphi_s (\bsx)\rd\bsx\\
&=\sum_{r_1=0}^{t_1}\cdots\sum_{r_s=0}^{t_s} \left(\prod_{j=1}^s
\left(\frac{t_j !}{T_j !}\right)^{1/2} {T_j \choose t_j -r_j} \frac{((\abs{h_j-v_j}+2r_j)!)^{1/2}}{r_j !}\right)\\
&\qquad\hspace{2cm}\times\int_{\RR^s} f(\bsx)\left(\prod_{j=1}^s
H_{\abs{h_j-v_j}+2r_j}(x_j)\right)\varphi_s (\bsx)\rd\bsx.
\end{align*}
For $j\in\{1,\ldots,s\}$, and given $v_j,h_j$, and $r_j$, we now
write
\begin{align*}
h_j\oplus_{r_j} v_j:= \abs{h_j-v_j}+2r_j,
\end{align*}
and by $\bsh\oplus_{\bsr} \bsv$ we denote the same operation
applied component-wise to vectors. With this notation,
\begin{align*}
\widehat{f_{\bsh}}(\bsv)&=\sum_{r_1=0}^{t_1}\cdots\sum_{r_s=0}^{t_s} \left(\prod_{j=1}^s
\left(\frac{t_j !}{T_j !}\right)^{1/2} {T_j \choose t_j -r_j} \frac{((h_j\oplus_{r_j} v_j)!)^{1/2}}{r_j !}\right)\\
&\qquad\hspace{2cm}\times\int_{\RR^s} f(\bsx) H_{\bsh\oplus_{\bsr}\bsv}(\bsx)\varphi_s (\bsx)\rd\bsx\\
&=\sum_{r_1=0}^{t_1}\cdots\sum_{r_s=0}^{t_s} \left(\prod_{j=1}^s
\left(\frac{t_j !}{T_j !}\right)^{1/2} {T_j \choose t_j -r_j}
\frac{((h_j\oplus_{r_j} v_j)!)^{1/2}}{r_j
!}\right)\widehat{f}(\bsh\oplus_{\bsr}\bsv).
\end{align*}
Therefore, from~\eqref{eqintfh},
\begin{align*}
&\left|\int_{\RR^s} f_{\bsh}(\bsx)\varphi_s(\bsx)\rd\bsx-Q_{n,s}(f_{\bsh}) \right|^2\\
&\quad\le \Bigg(\sum_{\bsv\in \cG_{n,s}^{\bot} \setminus \{\bszero\}}
(\sqrt[4]{8\pi})^{\abs{\bsv}_{\ast}} \Bigg|\sum_{r_1=0}^{t_1}\cdots\sum_{r_s=0}^{t_s}
\left(\prod_{j=1}^s \left(\frac{t_j !}{T_j !}\right)^{1/2} {T_j \choose t_j -r_j} \frac{((h_j\oplus_{r_j} v_j)!)^{1/2}}{r_j !}\right)\\
&\qquad \hspace{6cm}\times\widehat{f}(\bsh\oplus_{\bsr}\bsv)\Bigg|\Bigg)^2\\
&\quad\le \Bigg(\sum_{\bsv\in \cG_{n,s}^{\bot} \setminus \{\bszero\}}
(\sqrt[4]{8\pi})^{\abs{\bsv}_{\ast}} \sum_{r_1=0}^{t_1}\cdots\sum_{r_s=0}^{t_s} \left(\prod_{j=1}^s
\left(\frac{t_j !}{T_j !}\right)^{1/2} {T_j \choose t_j -r_j} \frac{((h_j\oplus_{r_j} v_j)!)^{1/2}}{r_j !}\right)\\
&\qquad\hspace{6cm}\times |\widehat{f}(\bsh\oplus_{\bsr}\bsv)| \omega_{\bsh\oplus_{\bsr}\bsv}^{-1/2}\omega_{\bsh\oplus_{\bsr}\bsv}^{1/2}\Bigg)^2.\\
\end{align*}
Using the Cauchy-Schwarz inequality we obtain
\begin{align}\label{errfhab}
&\left|\int_{\RR^s} f_{\bsh}(\bsx)\varphi_s(\bsx)\rd\bsx-Q_{n,s}(f_{\bsh})\right|^2\nonumber \\
&\quad\le  \left(\sum_{\bsv\in \cG_{n,s}^{\bot} \setminus \{\bszero\}}\sum_{r_1=0}^{t_1}
\cdots\sum_{r_s=0}^{t_s}|\widehat{f}(\bsh\oplus_{\bsr}\bsv)|^2\omega_{\bsh\oplus_{\bsr}\bsv}^{-1}\right)\nonumber\\
&\qquad\times \left(\sum_{\bsv\in \cG_{n,s}^{\bot} \setminus \{\bszero\}}
(\sqrt{8\pi})^{\abs{\bsv}_{\ast}}\sum_{r_1=0}^{t_1}\cdots\sum_{r_s=0}^{t_s}\left(\prod_{j=1}^s
\frac{t_j !}{T_j !} {T_j \choose t_j -r_j}^2 \frac{(h_j\oplus_{r_j} v_j)!}{(r_j !)^2}\right)\omega_{\bsh\oplus_{\bsr}\bsv}\right)\nonumber\\
&\quad= \Theta_1 \times \Theta_2,
\end{align}
where
\begin{align}\label{eqpoint5}
\Theta_1:=\sum_{r_1=0}^{t_1}\cdots\sum_{r_s=0}^{t_s}\sum_{\bsv\in
\cG_{n,s}^{\bot} \setminus
\{\bszero\}}|\widehat{f}(\bsh\oplus_{\bsr}\bsv)|^2\omega_{\bsh\oplus_{\bsr}\bsv}^{-1},
\end{align}
and
\begin{align*}
\Theta_2:=\sum_{\bsv\in \cG_{n,s}^{\bot} \setminus
\{\bszero\}}(\sqrt{8\pi})^{\abs{\bsv}_{\ast}}\sum_{r_1=0}^{t_1}\cdots\sum_{r_s=0}^{t_s}\left(\prod_{j=1}^s
\frac{t_j !}{T_j !} {T_j \choose t_j -r_j}^2
\frac{(h_j\oplus_{r_j} v_j)!}{(r_j
!)^2}\right)\omega_{\bsh\oplus_{\bsr}\bsv}.
\end{align*}
Now we estimate $\Theta_1$ and $\Theta_2$ from above.

\paragraph{Upper bound on $\Theta_1$:}

For given $\bsh=(h_1,\ldots,h_s)\in\A (s,M)$,
$\bsk=(k_1,\ldots,k_s)\in\NN_0^s$, and
\begin{align*}
\bsr=(r_1,\ldots,r_s)\in\bigotimes_{j=1}^s \{0,\ldots,t_j\},
\end{align*}
the system of equations
\begin{align*}
h_1\oplus_{r_1} v_1 &= k_1,\\
h_2\oplus_{r_2} v_2 &= k_2,\\
\vdots \quad&\\
h_s\oplus_{r_s} v_s &= k_s\\
\end{align*}
has at most $2^s$ solutions $(v_1,\ldots,v_s)\in \cG_{n,s}^{\bot}
\setminus \{\bszero\}$. Hence,
\begin{align*}
\Theta_1 \le \sum_{r_1=0}^{t_1}\cdots\sum_{r_s=0}^{t_s} 2^s
\sum_{\bsk\in\NN_0^s} |\widehat{f} (\bsk)|^2 \omega_{\bsk}^{-1}
\le 2^s \norm{f}_{K_{s,\bsa,\bsb,\omega}}^2 \prod_{j=1}^s (h_j
+1),
\end{align*}
where we used that $t_j \le h_j$.

Note that $\bsh\in\A (s,M)$ means by definition that
$\omega_{\bsh}^{-1}<M$, and this implies $\omega^{-a_j
h_j^{b_j}}<M$ for each $j\in\{1,\ldots,s\}$. Hence we obtain, for
$j\in\{1,\ldots,s\}$,
\begin{align*}
h_j\le \left\lceil \left(\frac{\log M}{a_j
\log\omega^{-1}}\right)^{1/b_j}\right\rceil -1,
\end{align*}
and so
\begin{align*}
\prod_{j=1}^s (h_j +1)\leq \prod_{j=1}^{s}
\left\lceil \left(\frac{\log M}{a_j
\log\omega^{-1}}\right)^{1/b_j}\right\rceil=
 \prod_{j=1}^{\min (s,j(x_M))}
\left\lceil \left(\frac{\log M}{a_j
\log\omega^{-1}}\right)^{1/b_j}\right\rceil,
\end{align*}
where $x_M=\log M/(\log \omega^{-1})$, and $j(x_M)=\sup\{j\in \NN:
x_M> a_j\}$. Overall we have
\begin{align}\label{estTheta1}
\Theta_1 \le 2^s \norm{f}_{K_{s,\bsa,\bsb,\omega}}^2
\prod_{j=1}^{\min (s,j(x_M))} \left\lceil \left(\frac{\log M}{a_j
\log\omega^{-1}}\right)^{1/b_j}\right\rceil.
\end{align}

\paragraph{Upper bound on $\Theta_2$:} Note that $h_j\oplus_{r_j} v_j=T_j-t_j+2r_j$ and therefore
\begin{align*}
\frac{t_j !}{T_j !} {T_j \choose t_j -r_j}^2
\frac{(h_j\oplus_{r_j} v_j)!}{(r_j !)^2}={t_j \choose r_j}{T_j -t_j
+2r_j \choose r_j}{T_j \choose t_j-r_j}.
\end{align*}
Hence,
\begin{align*}
\Theta_2=\sum_{\bsv\in \cG_{n,s}^{\bot} \setminus
\{\bszero\}}(\sqrt{8\pi})^{\abs{\bsv}_*} \prod_{j=1}^s
\sum_{r_j=0}^{t_j} {t_j \choose r_j}{T_j -t_j +2r_j \choose
r_j}{T_j \choose t_j-r_j}  \omega_{h_j\oplus_{r_j} v_j}.
\end{align*}
Since $a_j,b_j\ge 1$, we have\footnote{Here we require $b_j \ge 1$ since for $b_j \in (0,1)$ we would have, according to Jensen's inequality, that $(\abs{h_j-v_j}+2r_j)^{b_j} \le \abs{h_j-v_j}^{b_j}+(2r_j)^{b_j}$.} \label{bjgr1}
\begin{align*}
\omega_{h_j\oplus_{r_j} v_j}=\omega^{a_j
(\abs{h_j-v_j}+2r_j)^{b_j}}\le \omega^{a_j \abs{h_j-v_j}^{b_j} +
a_j (2r_j)^{b_j}} =\omega_{\abs{h_j -v_j}}\omega_{2r_j}\le
\omega_{\abs{h_j-v_j}}\omega^{2r_j}.
\end{align*}
Thus,
\begin{align*}
\Theta_2 &\le\sum_{\bsv\in \cG_{n,s}^{\bot} \setminus \{\bszero\}}(\sqrt{8\pi})^{\abs{\bsv}_*}
\prod_{j=1}^s  \omega_{\abs{h_j-v_j}}
\sum_{r_j=0}^{t_j}\omega^{2r_j} {t_j \choose r_j}{T_j -t_j +2r_j \choose r_j}{T_j \choose t_j-r_j}\\
&\le \sum_{\bsv\in \cG_{n,s}^{\bot} \setminus \{\bszero\}}(\sqrt{8\pi})^{\abs{\bsv}_*} \prod_{j=1}^s
\omega_{\abs{h_j-v_j}} \left(\sum_{r_j=0}^{t_j} {t_j \choose r_j}\omega^{r_j}\right)
\left(\sum_{r_j=0}^{t_j}{T_j +t_j  \choose r_j}\omega^{r_j}\right)\left(\sum_{r_j=0}^{t_j}{T_j \choose t_j-r_j}\right)\\
&= \sum_{\bsv\in \cG_{n,s}^{\bot} \setminus \{\bszero\}}(\sqrt{8\pi})^{\abs{\bsv}_*} \prod_{j=1}^s
\omega_{\abs{h_j-v_j}} \left(\sum_{r_j=0}^{t_j} {t_j \choose r_j}\omega^{r_j}\right)
\left(\sum_{r_j=0}^{t_j}{T_j +t_j  \choose r_j}\omega^{r_j}\right)\left(\sum_{r_j=0}^{t_j}{T_j \choose r_j}\right)\\
&\le \sum_{\bsv\in \cG_{n,s}^{\bot} \setminus
\{\bszero\}}(\sqrt{8\pi})^{\abs{\bsv}_*} \prod_{j=1}^s
\omega_{\abs{h_j-v_j}} \omega^{-t_j}\left(\sum_{r_j=0}^{t_j} {t_j
\choose r_j}\omega^{r_j} \right)\left(\sum_{r_j=0}^{t_j}{T_j +t_j
\choose r_j}\omega^{r_j}\right)\left(\sum_{r_j=0}^{t_j}{T_j
\choose r_j}\omega^{r_j}\right).
\end{align*}
Now, let $k=k(\omega)$ be the smallest positive integer such that
\begin{align*}
k\ge \frac{\log (\omega^{-1/8}-1)}{\log\omega}.
\end{align*}
We then get
\begin{align*}
\Theta_2 &\le \sum_{\bsv\in \cG_{n,s}^{\bot} \setminus \{\bszero\}}(\sqrt{8\pi})^{\abs{\bsv}_*} \prod_{j=1}^s
\Bigg[ \omega_{\abs{h_j-v_j}} \omega^{-t_j}\left(\sum_{r_j=0}^{t_j} {t_j \choose r_j}\omega^{kr_j}\omega^{-(k-1)r_j}\right)\\
&\quad\times \left(\sum_{r_j=0}^{t_j}{T_j +t_j  \choose r_j}\omega^{kr_j} \omega^{-(k-1)r_j}\right)
\left(\sum_{r_j=0}^{t_j}{T_j \choose r_j}\omega^{kr_j} \omega^{-(k-1)r_j}\right)\Bigg]\\
&\le \sum_{\bsv\in \cG_{n,s}^{\bot} \setminus \{\bszero\}}(\sqrt{8\pi})^{\abs{\bsv}_*} \prod_{j=1}^s
\Bigg[ \omega_{\abs{h_j-v_j}} \omega^{-(3k-2)t_j}\\
&\quad\times \left(\sum_{r_j=0}^{t_j} {t_j \choose r_j}\omega^{kr_j}\right)\left(\sum_{r_j=0}^{t_j}{T_j+t_j
\choose r_j}\omega^{kr_j}\right)\left(\sum_{r_j=0}^{t_j}{T_j \choose r_j}\omega^{kr_j}\right)\Bigg]\\
&\le \sum_{\bsv\in \cG_{n,s}^{\bot} \setminus \{\bszero\}}(\sqrt{8\pi})^{\abs{\bsv}_*} \prod_{j=1}^s
\Bigg[ \omega_{\abs{h_j-v_j}} \omega^{-(3k-2)t_j}\\
&\quad\times \left(\sum_{r_j=0}^{t_j} {t_j \choose r_j}\omega^{kr_j}\right)
\left(\sum_{r_j=0}^{T_j+t_j}{T_j+t_j \choose r_j}\omega^{kr_j}\right)\left(\sum_{r_j=0}^{T_j}{T_j \choose r_j}\omega^{kr_j}\right)\Bigg].
\end{align*}
Using the binomial theorem we obtain
\begin{align*}
\Theta_2 &\le \sum_{\bsv\in \cG_{n,s}^{\bot} \setminus \{\bszero\}}(\sqrt{8\pi})^{\abs{\bsv}_*} \prod_{j=1}^s
\omega_{\abs{h_j-v_j}} \omega^{-(3k-2)t_j} (1+\omega^k)^{2T_j +2t_j}\\
&=\sum_{\bsv\in \cG_{n,s}^{\bot} \setminus
\{\bszero\}}(\sqrt{8\pi})^{\abs{\bsv}_*} \prod_{j=1}^s
\omega_{\abs{h_j-v_j}} \omega^{-(3k-2)t_j}
(1+\omega^k)^{2h_j+2v_j}.
\end{align*}
Using again $t_j=\min (v_j,h_j)\le h_j$, we conclude
\begin{align*}
\Theta_2 &\le\sum_{\bsv\in \cG_{n,s}^{\bot} \setminus \{\bszero\}}(\sqrt{8\pi})^{\abs{\bsv}_*} \prod_{j=1}^s
\omega_{\abs{h_j-v_j}} \omega^{-(3k-2)h_j} (1+\omega^k)^{2h_j+2v_j}\\
&= \sum_{\bsv\in \cG_{n,s}^{\bot} \setminus
\{\bszero\}}(\sqrt{8\pi})^{\abs{\bsv}_*} \prod_{j=1}^s
\omega_{\abs{h_j-v_j}} \omega^{-(3k-2)h_j} \omega^{-(2 h_j +2
v_j)\frac{\log (1+\omega^k)}{\log \omega^{-1}}}.
\end{align*}
We now use
\begin{align*}
\abs{v_j}^{b_j}\le 2^{b_j} (\abs{v_j\pm
h_j}^{b_j}+\abs{h_j}^{b_j}),
\end{align*}
i.e.,
\begin{align*}
\frac{\abs{v_j}^{b_j}}{2^{b_j}}-\abs{h_j}^{b_j}\le \abs{v_j\pm
h_j}^{b_j}
\end{align*}
for any $b_j\ge 1$ and any $v_j, h_j\in\ZZ$. Consequently,
\begin{align*}
\Theta_2 &\le \sum_{\bsv\in \cG_{n,s}^{\bot} \setminus \{\bszero\}}(\sqrt{8\pi})^{\abs{\bsv}_*} \prod_{j=1}^s
\omega^{2^{-b_j} a_j v_j^{b_j}} \omega^{-a_j h_j^{b_j}}
\omega^{-h_j\left((3k-2)+\frac{2\log (1+\omega^k)}{\log \omega^{-1}}\right)}\omega^{-v_j \frac{2\log (1+\omega^k)}{\log \omega^{-1}}}\\
&\le\sum_{\bsv\in \cG_{n,s}^{\bot} \setminus \{\bszero\}}(\sqrt{8\pi})^{\abs{\bsv}_*}
\prod_{j=1}^s \omega^{2^{-b_j} a_j v_j^{b_j}}
\omega^{-a_j h_j^{b_j}\left(3k-1+\frac{2\log (1+\omega^k)}{\log \omega^{-1}}\right)}\omega^{-v_j \frac{2\log (1+\omega^k)}{\log \omega^{-1}}}\\
&=\sum_{\bsv\in \cG_{n,s}^{\bot} \setminus \{\bszero\}}
(\sqrt{8\pi})^{\abs{\bsv}_*}(\omega_{\bsh}^{-1})^{3k-1+\frac{2\log (1+\omega^k)}{\log \omega^{-1}}}
\prod_{j=1}^s\omega^{2^{-b_j} a_j v_j^{b_j}-v_j \frac{2\log (1+\omega^k)}{\log \omega^{-1}}}\\
&\le \sum_{\bsv\in \cG_{n,s}^{\bot} \setminus
\{\bszero\}}(\sqrt{8\pi})^{\abs{\bsv}_*} M^K
\prod_{j=1}^s\omega^{2^{-b_j} a_j v_j^{b_j}-v_j \frac{2\log
(1+\omega^k)}{\log \omega^{-1}}},
\end{align*}
where we used $a_j,b_j\ge 1$ \label{bjgr1(2)} for the second inequality, and
$\bsh\in\A (s,M)$ for the last inequality, and where
$K=K(\omega):= 3k-1+\frac{2\log (1+\omega^k)}{\log \omega^{-1}}$.

For $j\in\{1,\ldots,s\}$ we now study the term
\begin{align*}
\omega^{2^{-b_j} a_j v_j^{b_j}-v_j \frac{2\log (1+\omega^k)}{\log
\omega^{-1}}}=\omega^{2^{-b_j-1} a_j v_j^{b_j}}\omega^{2^{-b_j-1}
a_j v_j^{b_j}-v_j \frac{2\log (1+\omega^k)}{\log \omega^{-1}}}.
\end{align*}
We show that
\begin{align}\label{eqexpge0}
2^{-b_j-1} a_j v_j^{b_j}-v_j \frac{2\log (1+\omega^k)}{\log
\omega^{-1}}\ge 0.
\end{align}
Indeed,~\eqref{eqexpge0} is trivially fulfilled if $v_j=0$. If
$v_j>0$, this implies $v_j\ge 2$, since $\bsv\in \cG_{n,s}^{\bot}$.
In this case,~\eqref{eqexpge0} is fulfilled if and only if
\begin{align}\label{eqexpge0b}
v_j^{b_j-1} \ge 8\frac{1}{a_j} \frac{\log (1+\omega^k)}{\log
\omega^{-1}}2^{b_j-1}.
\end{align}
Since $v_j\ge 2$, and since $a_j\ge 1$, \eqref{eqexpge0b} is
certainly fulfilled if
\begin{align*}
2^{b_j-1}\ge 8 \frac{\log (1+\omega^k)}{\log
\omega^{-1}}2^{b_j-1}.
\end{align*}
However, $k$ was chosen exactly such that the latter condition
holds true. Hence,~\eqref{eqexpge0} is satisfied, and we have
\begin{align*}
\omega^{2^{-b_j} a_j v_j^{b_j}-v_j \frac{2\log (1+\omega^k)}{\log
\omega^{-1}}}\le\omega^{2^{-b_j-1} a_j v_j^{b_j}}.
\end{align*}

Consequently,
\begin{align*}
\Theta_2 \le M^K \sum_{\bsv\in \cG_{n,s}^{\bot} \setminus
\{\bszero\}}(\sqrt{8\pi})^{\abs{\bsv}_*}\omega^{\frac{1}{2}\sum_{j=1}^s
a_j (v_j/2)^{b_j}}.
\end{align*}

Now we insert our upper bounds for $\Theta_1$ and $\Theta_2$ into
\eqref{errfhab}. For $f$ in the unit ball of
$\cH(K_{s,\bsa,\bsb,\omega})$ we obtain
\begin{align*}
&\left|\int_{\RR^s} f_{\bsh}(\bsx)\varphi_s(\bsx)\rd\bsx-Q_{n,s}(f_{\bsh}) \right|^2\\
&\quad \le  2^s \left(\prod_{j=1}^{\min (s,j(x_M))} \left\lceil
\left(\frac{\log M}{a_j
\log\omega^{-1}}\right)^{1/b_j}\right\rceil\right)M^K
\sum_{\bsv\in \cG_{n,s}^{\bot} \setminus
\{\bszero\}}(\sqrt{8\pi})^{\abs{\bsv}_*}\omega^{\frac{1}{2}\sum_{j=1}^s
a_j (v_j/2)^{b_j}},
\end{align*}
as claimed.
\end{proof}

Next we show the following proposition.
\begin{prop}
We have
\begin{align}\label{eqboundMapp2}
[e^{{\rm
app}}(\mathcal{H}(K_{s,\bsa,\bsb,\omega}),A_{n,s,M})]^2\le\frac{1}{M}+
M^{2B(s)+K} D(s,\omega,\bsb) F_n,
\end{align}
where $B(s):=\sum_{j=1}^s b_j^{-1}$, $K=K(\omega)$ as in
\eqref{def_K},
\begin{align}\label{def_Fn}
F_n:= \sum_{\bsv\in \cG_{n,s}^{\bot} \setminus
\{\bszero\}}(\sqrt{8\pi})^{\abs{\bsv}_*}\omega^{\frac{1}{2}\sum_{j=1}^s
a_j (v_j/2)^{b_j}},
\end{align}
and where
\begin{align}\label{def_D}
D(s,\omega,\bsb):=8^s \prod_{j=1}^s\left(1+
\log^{-1/b_j}\omega^{-1}\right)^2.
\end{align}
\end{prop}

\begin{proof}
Let $f \in \cH(K_{s,\bsa,\bsb,\omega})$ with
$\norm{f}_{K_{s,\bsa,\bsb,\omega}}\le 1$. Using
\eqref{approx_err}, \eqref{approx_err1}, and Lemma \ref{le5}, we
have
\begin{align*}
\norm{f-A_{n,s,M}(f)}_{\LL_2}^2 &\le \frac{1}{M} + \sum_{\bsh\in\A (s,M)} 2^s \left(\prod_{j=1}^{\min (s,j(x_M))}
\left\lceil \left(\frac{\log M}{a_j \log\omega^{-1}}\right)^{1/b_j}\right\rceil\right)\\
&\quad\times M^K \sum_{\bsv\in \cG_{n,s}^{\bot} \setminus \{\bszero\}}(\sqrt{8\pi})^{\abs{\bsv}_*}\omega^{\frac{1}{2}\sum_{j=1}^s a_j (v_j/2)^{b_j}}\\
&= \frac{1}{M} + \abs{\A (s,M)} 2^s \left(\prod_{j=1}^{\min (s,j(x_M))} \left\lceil
\left(\frac{\log M}{a_j \log\omega^{-1}}\right)^{1/b_j}\right\rceil\right)\\
&\quad\times M^K \sum_{\bsv\in \cG_{n,s}^{\bot} \setminus \{\bszero\}}(\sqrt{8\pi})^{\abs{\bsv}_*}\omega^{\frac{1}{2}\sum_{j=1}^s a_j (v_j/2)^{b_j}}.
\end{align*}
Since $\vert\A(s,M)\vert\leq\prod_{j=1}^s(1+(\log M/(a_j\log\omega^{-1}))^{1/b_j})$ due to \cite[Lemma 1]{IKPW15} we have
\begin{align*}
\norm{f-A_{n,s,M}(f)}_{\LL_2}^2 &\le \frac{1}{M} + 2^s M^K\left(\prod_{j=1}^{\min (s,j(x_M))} \left\lceil
\left(\frac{\log M}{a_j \log\omega^{-1}}\right)^{1/b_j}\right\rceil\right)\\
&\quad\times  \left(\prod_{j=1}^s \left(1+\left(\frac{\log M}{a_j
\log\omega^{-1}}\right)^{1/b_j}\right)\right)\sum_{\bsv\in
\cG_{n,s}^{\bot} \setminus
\{\bszero\}}(\sqrt{8\pi})^{\abs{\bsv}_*}\omega^{\frac{1}{2}\sum_{j=1}^s
a_j (v_j/2)^{b_j}}.
\end{align*}

This means that
\begin{align}\label{eqboundMapp}
[e^{{\rm app}}(\mathcal{H}(K_{s,\bsa,\bsb,\omega}),
A_{n,s,M})]^2\le \frac{1}{M}+ 2^s M^K
\left(\prod_{j=1}^s\left(1+\left(\frac{\log M}{a_j \log
\omega^{-1}}\right)^{1/b_j}\right)\right)^2F_n,
\end{align}
where $F_n$ is as in \eqref{def_Fn}.

Furthermore, we estimate
\begin{align*}
\prod_{j=1}^s\left(1+\left(\frac{\log M}{a_j \log \omega^{-1}}\right)^{1/b_j}\right)
&\le\prod_{j=1}^s\left(1+\left(\frac{\log M}{ \log \omega^{-1}}\right)^{1/b_j}\right)\\
&\le \prod_{j=1}^s\left(1+
\log^{-1/b_j}\omega^{-1}\right)\prod_{j=1}^s\left(1+\log^{1/b_j}M\right).
\end{align*}
Since $M$ is assumed to be at least 1, we can bound
$1+\log^{1/b_j} M\le 2M^{1/b_j}$, and obtain
\begin{align*}
\prod_{j=1}^s\left(1+\left(\frac{\log M}{a_j\log
\omega^{-1}}\right)^{1/b_j}\right)\le 2^s
M^{B(s)}\prod_{j=1}^s\left(1+ \log^{-1/b_j} \omega^{-1}\right).
\end{align*}
Plugging this into \eqref{eqboundMapp}, we obtain
\begin{align*}
[e^{{\rm
app}}(\mathcal{H}(K_{s,\bsa,\bsb,\omega}),A_{n,s,M})]^2\le\frac{1}{M}+
M^{2B(s)+K} D(s,\omega,\bsb) F_n,
\end{align*}
where $D(s,\omega,\bsb)$ is as in \eqref{def_D}.
\end{proof}

We now give the proof of Lemma \ref{le_Hermitseriesfh}:

\begin{proof}
To show that $f_{\bsh}$ can be pointwise represented by its
Hermite series, due to \cite[Proposition~2.6]{IL} it is sufficient to verify that
\begin{align*}
\sum_{\bsv\in\NN_0^s}|\widehat{f_{\bsh}}(\bsv)|<\infty.
\end{align*}
To this end we proceed quite similarly to what we did when we
estimated
\begin{align*}
\left(\sum_{\bsv\in \cG_{n,s}^{\bot} \setminus \{\bszero\}}
|\widehat{f_{\bsh}}(\bsv)| (\sqrt[4]{8\pi})^{\abs{\bsv}_*}
\right)^2,
\end{align*}
see \eqref{eqintfh}, in the proof of Lemma~\ref{le5}. By going
through analogous steps, we see that
\begin{align*}
\left(\sum_{\bsv\in\NN_0^s}|\widehat{f_{\bsh}}(\bsv)|\right)^2 &\le \norm{f}_{K_{s,\bsa,\bsb,\omega}}^2 2^s
\left(\prod_{j=1}^{\min (s,j(x_M))} \left\lceil \left(\frac{\log M}{a_j \log\omega^{-1}}\right)^{1/b_j}\right\rceil\right) M^K\\
&\quad\times\sum_{\bsv\in \NN_0^s} \prod_{j=1}^s\omega^{2^{-b_j}
a_j v_j^{b_j}-v_j \frac{2\log (1+\omega^k)}{\log \omega^{-1}}}.
\end{align*}
However,
\begin{align*}
\sum_{\bsv\in \NN_0^s}
\prod_{j=1}^s\omega^{2^{-b_j} a_j v_j^{b_j}-v_j \frac{2\log (1+\omega^k)}{\log \omega^{-1}}}&=\prod_{j=1}^s
\sum_{v_j=0}^\infty\omega^{2^{-b_j} a_j v_j^{b_j}-v_j \frac{2\log (1+\omega^k)}{\log \omega^{-1}}}\\
&=\prod_{j=1}^s \left(1+ \omega^{2^{-b_j} a_j -\frac{2\log
(1+\omega^k)}{\log \omega^{-1}}}
+\sum_{v_j=2}^\infty\omega^{2^{-b_j} a_j v_j^{b_j}-v_j \frac{2\log
(1+\omega^k)}{\log \omega^{-1}}}\right).
\end{align*}
In the derivation of~\eqref{eqexpge0}, it was sufficient that
$v_j\ge 2$. Hence we can proceed analogously for the sum in the
latter expression to see that this sum is finite. Hence we
derive that
\begin{align*}
\left(\sum_{\bsv\in\NN_0^s}|\widehat{f_{\bsh}}(\bsv)|\right)^2<\infty.
\end{align*}
\end{proof}

\subsection{The proof of Theorem \ref{th:result}}\label{prThm1}

We now prove Theorem \ref{th:result}. To this end, we need the
following proposition.

\begin{prop}\label{prop_upper_expo_conv}
For $s\in \NN$ and $\e\in(0,1)$ define
\begin{align*}
m=\max_{j=1,2,\dots,s}\ \left\lceil
\left(\frac{2^{b_j+1}}{a_j}\,
\frac{\log\left(1+\frac{s \sqrt{8\pi}}{(1-\omega^{1/2}) \log(1+\eta^2)}\right)}{\log\,\omega^{-1}}\right)^{B(s)}\,\right\rceil,
\end{align*}
where
\begin{align*}
\eta=\left(\frac{\e^2}{2D(s,\omega,\bsb)^{\frac{1}{2B(s)+K+1}}}
\right)^{\frac{2B(s)+K+1}{2}}
\end{align*}
and $K=K(\omega)$ as in \eqref{def_K}. Let $m_1,m_2,\ldots,m_s$ be given by
\begin{align*}
m_j:=\left\lfloor m^{1/(B(s)  b_j)}\right\rfloor\ \ \ \ \
\mbox{for}\ \ \  j=1,2,\ldots,s\ \ \ \mbox{and}\ \ \
n=\prod_{j=1}^s m_j.
\end{align*}
Then for $M=2/\e^2$ we have
\begin{align*}
e^{{\rm app}}(\mathcal{H}(K_{s,\bsa,\bsb,\omega}),A_{n,s,M})\le\e\
\ \ \ \mbox{and}\ \ \ \
n=\mathcal{O}(\log^{\,B(s)}(1+\e^{-1}))
\end{align*}
with the factor in the $\mathcal{O}$ notation independent of
$\e^{-1}$ but dependent on $s$.
\end{prop}
\begin{proof}
From \eqref{def_Fn} we have
\begin{align*}
F_n&=\sum_{\bsv\in \cG_{n,s}^{\bot} \setminus \{\bszero\}}(\sqrt{8\pi})^{\abs{\bsv}_*}\omega^{\frac{1}{2}\sum_{j=1}^s a_j (v_j/2)^{b_j}}\\
&=-1+\prod_{j=1}^s \left(1+\sqrt{8\pi} \sum_{h=m_j}^\infty \omega^{\frac{1}{2} a_j h^{b_j}}\right)\\
&=-1+\prod_{j=1}^s \left(1+\omega^{\frac{1}{2} a_j m_j^{b_j}}\sqrt{8\pi} \sum_{h=m_j}^\infty \omega^{\frac{1}{2} a_j (h^{b_j}-m_j^{b_j}) }\right)\\
& \le-1+\prod_{j=1}^s \left(1+\omega^{\frac{1}{2} a_j m_j^{b_j}}\sqrt{8\pi} \sum_{h=0}^\infty \omega^{\frac{1}{2} h}\right)\\
&=-1+\prod_{j=1}^s \left(1+\omega^{\frac{1}{2} a_j
m_j^{b_j}}\frac{\sqrt{8\pi}}{1-\omega^{1/2}}\right),
\end{align*}
where we used that $a_j (h^{b_j}-m_j^{b_j}) \ge h-m_j$, since $a_j,b_j \ge 1$. 

Since $\lfloor x\rfloor\ge x/2$ for all $x\ge1$, we have
\begin{align*}
(2m_j)^{b_j}\ge m^{1/B(s)}\qquad\mbox{for all}\qquad
j=1,2,\dots,s.
\end{align*}
Hence,
\begin{align*}
F_n \le -1+\prod_{j=1}^s \left(1+\omega^{m^{1/B(s)} a_j
2^{-b_j-1}}\frac{\sqrt{8\pi}}{1-\omega^{1/2}}\right).
\end{align*}

From the definition of $m$ we have
\begin{align*}
\omega^{m^{1/B(s)} a_j
2^{-b_j-1}}\frac{\sqrt{8\pi}}{1-\omega^{1/2}}\le
\frac{\log(1+\eta^2)}{s}\qquad\mbox{for all}\qquad j=1,2,\dots,s.
\end{align*}
This proves
\begin{align}\label{eqFnbound}
F_n\le-1+\left(1+\frac{\log(1+\eta^2)}{s}\right)^s\le
-1+\exp(\log(1+\eta^2))=\eta^2.
\end{align}
Now, plugging this into \eqref{eqboundMapp2}, we obtain
\begin{align}\label{eqboundMapp3}
[e^{{\rm
app}}(\mathcal{H}(K_{s,\bsa,\bsb,\omega}),A_{n,s,M})]^2\le\frac{1}{M}+
M^{2B(s)+K} D(s,\omega,\bsb) \eta^2.
\end{align}
Note that
\begin{align*}
(D(s,\omega,\bsb)\eta^2)^{-\frac{1}{2B(s)+K+1}}=\frac{2}{\e^2}\ge 1.
\end{align*}
Hence we are allowed to choose
\begin{align*}
M=(D(s,\omega,\bsb)\eta^2)^{-\frac{1}{2B(s)+K+1}},
\end{align*}
which yields, inserting into \eqref{eqboundMapp3},
\begin{align*}
[e^{{\rm
app}}(\mathcal{H}(K_{s,\bsa,\bsb,\omega}),A_{n,s,M})]^2\le
2 (D(s,\omega,\bsb)\eta^2)^{\frac{1}{2B(s)+K+1}}=\e^2,
\end{align*}
as claimed.

It remains to verify that $n$ is of the order stated in the
proposition. Note that
\begin{align*}
n=\prod_{j=1}^s m_j=\prod_{j=1}^s \left\lfloor m^{1/(B(s) b_j)}
\right\rfloor \le m^{\frac{1}{B(s)}\sum_{j=1}^s1/b_j} = m.
\end{align*}
However,
\begin{align*}
m=\mathcal{O}(\log^{B(s)}(1 +\eta^{-1})),
\end{align*}
as $\eta$ tends to zero. From this, it is easy to see that we
indeed have
\begin{align*}
n=\mathcal{O}(\log^{B(s)}(1 +\e^{-1})),
\end{align*}
which concludes the proof of Proposition
\ref{prop_upper_expo_conv}.
\end{proof}

We now prove the successive points of Theorem \ref{th:result}.
\subsubsection*{Proof of Point \ref{exp} (Exponential Convergence)}

We conclude from  Proposition \ref{prop_upper_expo_conv} that
\begin{align*}
n(\e,\APP_s,\Lambda^{\rm{std}})=\mathcal{O}(\log^{B(s)}(1
+\e^{-1})).
\end{align*}
This implies that we indeed have EXP for all $\bsa$ and $\bsb$,
with $p(s)=1/B(s)$, and thus $p^* (s)\ge 1/B(s)$. On the other
hand, note that obviously $e(n,\APP_s,\lstd)\ge
e(n,\APP_s,\lall)$, hence the rate of EXP for $\Lambda^{\rm{std}}$
cannot be larger than for $\Lambda^{\rm{all}}$ which is $1/B(s)$.
Thus, we have $p^*(s)= 1/B(s)$.

\medskip
\subsubsection*{Proof of Point \ref{uexp} (Uniform Exponential Convergence)}

Suppose first that $\bsa$ is an arbitrary sequence and that $\bsb$
is such that
\begin{align*}
B=\sum_{j=1}^\infty \frac{1}{b_j}<\infty.
\end{align*}
Then we can replace $B(s)$ by $B$ in
Proposition~\ref{prop_upper_expo_conv}, and we obtain
\begin{align*}
n(\varepsilon,\APP_s,\lstd)=\mathcal{O}\left(\log^{B}\left (1
+\e^{-1}\right)\right),
\end{align*}
hence UEXP with $p^* \ge 1/B$ holds. On the other hand, if we have
UEXP for $\Lambda^{\rm{std}}$, this implies UEXP for
$\Lambda^{\rm{all}}$, which in turn implies that $B<\infty$ and
that $p^*  \le 1/B$.

\medskip
\subsubsection*{Proof of Point \ref{wt} (EC-Weak Tractability)}
Assume that EC-WT holds for the class $\Lambda^{\rm{std}}$. Then
EC-WT also holds for the class $\Lambda^{\rm{all}}$ and this
implies that $\lim_ja_j=\infty$, as claimed.

Assume now that $\lim_ja_j=\infty$. We consider the operator
\begin{align*}
W_s:=\APP_s^\ast \APP_s: \calH (K_{s,\bsa,\bsb,\omega})\To \calH
(K_{s,\bsa,\bsb,\omega}),
\end{align*}
which is given by
\begin{align*}
W_s f=\sum_{\bsk\in\NN_0^s}\omega_{\bsk} \langle
f,e_{\bsk}\rangle_{K_{s,\bsa,\bsb,\omega}} e_{\bsk} \ \ \ \mbox{
for } \ f\in\calH (K_{s,\bsa,\bsb,\omega}),
\end{align*}
where $e_{\bsk}=\sqrt{\omega_{\bsk}} H_{\bsk}$ and $\langle e_{\bsk},e_{\bsl}\rangle=\delta_{\bsk,\bsl}$. We then have
\begin{align*}
W_s e_{\bsk}=\omega_{\bsk}e_{\bsk}\quad\textnormal{ for all }\
\bsk\in\NN_0^s,
\end{align*}
so the eigenpairs of $W_s$ are $(\omega_{\bsk},e_{\bsk})$ for
$\bsk \in \NN_0^s$; see \cite[Section 3]{IKPW15} for more details.

We use \cite[Theorem~26.18]{NW12} which states that if the
ordered eigenvalues $\lambda_{s,n}$ of $W_s$ satisfy
\begin{equation}\label{2618}
\lambda_{s,n}\le \frac{M^{\,2}_{s,\tau}}{n^{2\tau}} \ \ \ \ \ \ \
\mbox{for all} \ \ \ \ \ n\in\NN,
\end{equation}
for some positive $M_{s,\tau}$ and $\tau>\tfrac12$ then there is a
semi-constructive algorithm\footnote{By semi-constructive we mean
that this algorithm can be constructed after a few random
selections of sample points, more can be found in~\cite[Section
24.3]{NW12}.} such that
\begin{equation}\label{2619}
e(n+2,\APP_s;\lstd)\le\frac{M_{s,\tau}\,C(\tau)}{n^{\tau(2\tau/(2\tau+1))}}\
\ \ \ \ \ \   \mbox{for all} \ \ \ \ \ n\in\NN
\end{equation}
where $C(\tau)$ is given explicitly in \cite[Theorem~26.18]{NW12}.
However, the form of $C(\tau)$ is not important for our
consideration.

For $\eta\in(0,1)$, let $\tau=1/(2\eta)>\tfrac12$. We stress that
$\tau$ can be arbitrarily large if we take sufficiently small
$\eta$. Now we have
\begin{align*}
n\lambda_{s,n}^{\eta}\le\sum_{j=1}^\infty\lambda_{s,j}^{\eta}=\sum_{\bsh\in\NN_0^s}
\omega_{\bsh}^{\eta}=\prod_{j=1}^s\left(1+\sum_{h=1}^\infty\omega^{\,\eta\,a_j\,h^{b_j}}\right).
\end{align*}
Note that
\begin{align*}
\sum_{h=1}^\infty\omega^{\,\eta\,a_j\,h^{b_j}}\le
\sum_{h=1}^\infty\omega^{\,\eta\,a_j\,h}=\omega^{\eta\,a_j}\sum_{h=1}^\infty\omega^{\,\eta\,a_j\,(h-1)}\le\omega^{\eta\,a_j}A_{\eta},
\end{align*}
where
\begin{align}\label{eqdefAeta}
A_\eta:= \sum_{h=0}^\infty
\omega^{\,\eta\,h}=\frac{1}{1-\omega^{\eta}}<\infty.
\end{align}
This proves that
\begin{align*}
\lambda_{s,n}\le\frac{\prod_{j=1}^s\left(1+\,\omega^{\eta\,a_j}A_\eta\right)^{1/\eta}}{n^{1/\eta}}.
\end{align*}
Hence, we can take
\begin{align*}
M_{s,\tau}=\prod_{j=1}^s\left(1+c_j\right)^{\tau}<\infty\quad\mbox{with}\quad
c_j=\omega^{\,a_j/(2\tau)}A_{\frac{1}{2\tau}},
\end{align*}
where $A_{\frac{1}{2\tau}}$ is defined as in \eqref{eqdefAeta}.
Furthermore, we know that $\lim_ja_j=\infty$ implies that
$\lim_s\sum_{j=1}^sc_j/s=0$.

From \eqref{2619} we obtain
$$
n(\varepsilon,\APP_s;\lstd)\le3+\left(M_{s,\tau}\,C(\tau)\right)^{(1+1/(2\tau))/\tau}\,\e^{-(1+1/(2\tau))/\tau}.
$$
This yields that
$$
\limsup_{s+ \e^{-1}\to\infty}\frac{\log\,n(\varepsilon,\APP_s;\lstd)}{s+\log\,\e^{-1}}\le
\left(1+\frac1{2\tau}\right)\,\frac1{\tau}\,\left(1+\limsup_{s\to\infty}\frac{\log\,M_{s,\tau}}{s}\right).
$$
Since $(\log M_{s,\tau})/s\le \tau\,\sum_{j=1}^sc_j/s$ tends to
zero as $s\rightarrow\infty$, we have
$$
\limsup_{s+\e^{-1}\to\infty}\frac{\log\,n(\varepsilon,\APP_s;\lstd)}{s+\log\,\e^{-1}}\le
\left(1+\frac1{2\tau}\right)\,\frac1{\tau}.
$$
Since $\tau$ can be arbitrarily large this proves that
$$
\lim_{s+\e^{-1}\to\infty}\frac{\log\,n(\varepsilon,\APP_s;\lstd)}{s+\log\,\e^{-1}}=0.
$$
This means that EC-WT holds for the class $\Lambda^{\rm{std}}$, as
claimed.

\medskip
\subsubsection*{Proof of Point \ref{pt} (EC-Polynomial Tractability)}
Suppose that EC-PT holds for the class $\Lambda^{\rm{std}}$. Then
EC-PT holds for the class $\Lambda^{\rm{all}}$. From~\cite{IKPW15}
we know that this implies EC-SPT for the class $\lall$ which is
equivalent to $B<\infty$ and $\alpha^*>0$. If the conditions
$B<\infty$ and $\alpha^*>0$ hold, we will show in the following
that this implies EC-SPT and therefore we also have EC-PT.

\medskip
\subsubsection*{Proof of Point \ref{spt} (EC-Strong Polynomial Tractability)}
The necessity of the conditions for EC-SPT on $\bsb$ and $\bsa$
follows from the same conditions for the class
$\Lambda^{\rm{all}}$ and the fact that the information complexity
for $\Lambda^{\rm{std}}$ cannot be smaller than for
$\Lambda^{\rm{all}}$.

To prove the sufficiency of the conditions for EC-SPT on $\bsa$
and $\bsb$ stated in Point \ref{spt} we analyze the algorithm
$A_{n,s,M}$ given by \eqref{eqdefAnsM}, where the sample points
$\bsx_k$ come from a Gauss-Hermite rule with
\begin{align*}
m_j=2\,\left\lceil\,
\left(\frac{\log\,M}{a_j^{\beta}\log\,\widetilde{\omega}^{-1}}\right)^{1/b_j}\right\rceil\,-\,1\qquad\mbox{for
all}\qquad j=1,2,\dots,s,
\end{align*}
where $M>1$, $\beta\in(0,1)$, and
$\widetilde{\omega}:=\omega^{\frac{1}{2K+2}}$ with $K=K(\omega)$,
defined in \eqref{def_K}. Note that $m_j\ge1$ and is always an
odd number. Furthermore $m_j=1$ if $a_j\ge ((\log M)/(\log
\widetilde{\omega}^{-1}))^{1/\beta}$. We know that $\alpha^* \in
(0,\infty]$. Since for all $\delta\in(0,\alpha^*)$ we have
\begin{align*}
a_j\ge \exp(\delta j) \quad\mbox{for all}\quad j\ge j^*_\delta,
\end{align*}
we conclude that
\begin{align*}
j\ge j^{*}_{\beta,\delta,M}:=\max\left(j^*_\delta,\frac{\log(((\log
M)/(\log\widetilde{\omega}^{-1}))^{1/\beta})}{\delta}\right)\quad
\mbox{implies}\quad m_j=1.
\end{align*}

For given $\bsh\in\A (s,M)$ and $\bsr\in\NN_0^s$ suppose that
\begin{align*}
\bsh\oplus_{\bsr}\bsv^{(1)}=\bsh\oplus_{\bsr}\bsv^{(2)}
\end{align*}
for some $\bsv^{(1)},\bsv^{(2)}\in \cG_{n,s}^{\bot} \setminus
\{\bszero\}$, $\bsv^{(1)}\neq \bsv^{(2)}$. This means that for all
$j\in\{1,\ldots,s\}$ we must have $|h_j-v_j^{(1)}|+2r_j =
|h_j-v_j^{(2)}|+2r_j$, which is equivalent to $|h_j-v_j^{(1)}| =
|h_j-v_j^{(2)}|$. As $\bsv^{(1)}\neq \bsv^{(2)}$, there must be at least one 
$j \in\{1,\ldots,s\}$ such that $v_j^{(1)}\neq v_j^{(2)}$. For this $j$, the condition 
$|h_j-v_j^{(1)}| =|h_j-v_j^{(2)}|$ is then equivalent to $2h_j=v_j^{(1)}+v_j^{(2)}$. From the
choice of $\bsh,\bsv^{(1)}$ and $\bsv^{(2)}$ it follows that for this $j$ we must have
 $$2 h_j=v_j^{(1)}+v_j^{(2)} \ge \max(v_j^{(1)},v_j^{(2)}) \ge 2 m_j$$ 
and hence for this $j$ we have $h_j \ge m_j$. This leads to a
contradiction, because if $h_j$ is the $j$th component of
$\bsh\in\A (s,M)$, we must have $$m_j \le h_j < \left(\frac{\log
M}{a_j \log\omega^{-1}}\right)^{1/b_j}\le \left(\frac{\log
M}{a_j^{\beta} \log\widetilde{\omega}^{-1}}\right)^{1/b_j}
\le\frac{m_j+1}{2} \le m_j$$ for each $j\in\{1,\ldots,s\}$.

Consequently, each coefficient $\widehat{f}
(\bsh\oplus_{\bsr}\bsv)$ occurs at most once in~\eqref{eqpoint5},
and so we get rid of the factor $2^s$ in the upper bound
\eqref{estTheta1} of $\Theta_1$. This way we obtain the improved
bound
\begin{align*}
\Theta_1 \le  \norm{f}_{K_{s,\bsa,\bsb,\omega}}^2
\prod_{j=1}^{\min (s,j(x_M))} \left\lceil \left(\frac{\log M}{a_j
\log\omega^{-1}}\right)^{1/b_j}\right\rceil.
\end{align*}
Together with our previous upper bounds on $\Theta_2$ we obtain
\begin{align*}
e_{n,s}^2:=&[e(A_{n,s,M},\APP_s;\lstd)]^2\\
\le & \frac1M\,+\,\sum_{\bsh\in\mathcal{A}(s,M)} \left(\prod_{j=1}^{\min (s,j(x_M))}
\left\lceil \left(\frac{\log M}{a_j \log\omega^{-1}}\right)^{1/b_j}\right\rceil\right)\\
&\times M^K \sum_{\bsv\in \cG_{n,s}^{\bot} \setminus
\{\bszero\}}(\sqrt{8\pi})^{\abs{\bsv}_*}\omega^{\frac{1}{2}\sum_{j=1}^s
a_j (v_j/2)^{b_j}}.
\end{align*}

For $s\in\NN$ we use the notation $[s]=\{1,\ldots,s\}$.
We now estimate
\begin{align*}
\sum_{\bsv\in \cG_{n,s}^{\bot} \setminus
\{\bszero\}}(\sqrt{8\pi})^{\abs{\bsv}_*}\omega^{\frac{1}{2}\sum_{j=1}^s
a_j (v_j/2)^{b_j}}=\sum_{\emptyset\neq\uu\subseteq
[s]}\prod_{j\in\uu} \sum_{\ell=0}^\infty
\sqrt{8\pi}\omega^{\frac{1}{2}a_j ((2m_j +2\ell)/2)^{b_j}},
\end{align*}
where we separated the cases for $v_j \neq 0$ and $v_j=0$.

Note that, as $m_j,b_j\ge 1$,\label{bjgr1(3)}
\begin{align*}
((2 m_j
+2\ell)/2)^{b_j}\ge\left(\frac{m_j+1}{2}+\ell\right)^{b_j}\ge\left(\frac{m_j+1}{2}\right)^{b_j}+\ell^{b_j}.
\end{align*}
Hence,
\begin{align*}
\sum_{\ell=0}^\infty \sqrt{8\pi}\omega^{\frac{1}{2}a_j ((2m_j +2\ell)/2)^{b_j}}
&\le\sum_{\ell=0}^\infty \sqrt{8\pi}\omega^{\frac{1}{2}a_j \left(\left(\frac{m_j+1}{2}\right)^{b_j}+\ell^{b_j}\right)}\\
&=\omega^{\frac{1}{2}a_j \left(\frac{m_j+1}{2}\right)^{b_j}} \sum_{\ell=0}^\infty\sqrt{8\pi}\omega^{\frac{1}{2}a_j \ell^{b_j}}\\
&\le \omega^{\frac{1}{2}a_j \left(\frac{m_j+1}{2}\right)^{b_j}}A,
\end{align*}
where $A=\frac{\sqrt{8 \pi}}{1-\sqrt{\omega}}$. Consequently,
\begin{eqnarray*}
\sum_{\bsv\in \cG_{n,s}^{\bot} \setminus \{\bszero\}}(\sqrt{8\pi})^{\abs{\bsv}_*}\omega^{\frac{1}{2}\sum_{j=1}^s a_j (v_j/2)^{b_j}}
& \le & \sum_{\emptyset\neq\uu\subseteq
[s]}\prod_{j\in\uu}\omega^{\frac{1}{2}a_j
\left(\frac{m_j+1}{2}\right)^{b_j}}A \\
& = & -1+\prod_{j=1}^s
\left(1+\omega^{\frac{1}{2}a_j
\left(\frac{m_j+1}{2}\right)^{b_j}}A\right).
\end{eqnarray*}

Furthermore,
\begin{align*}
e_{n,s}^2 \le & \frac{1}{M}+ M^K |\mathcal{A}(s,M)| \left(\prod_{j=1}^{\min (s,j(x_M))} \left\lceil
\left(\frac{\log M}{a_j \log\omega^{-1}}\right)^{1/b_j}\right\rceil\right)\\
&\times\left(-1 + \prod_{j=1}^s \left(1+\omega^{\frac{1}{2}a_j
\left(\frac{m_j+1}{2}\right)^{b_j}}A\right)\right).
\end{align*}
Using $\log (1+x) \le x$ we obtain
\begin{align*}
\log\left[ \prod_{j=1}^s \left(1+\omega^{\frac{1}{2}a_j
\left(\frac{m_j+1}{2}\right)^{b_j}}A \right)\right] \le  A
\sum_{j=1}^\infty \omega^{\frac{1}{2}a_j
\left(\frac{m_j+1}{2}\right)^{b_j}}=:\gamma.
\end{align*}
From the definition of $m_j$ we have $a_j[(m_j+1)/2]^{b_j}\ge
a_j^{1-\beta}\,(\log\,M)/\log\,\widetilde{\omega}^{-1}$. Therefore
\begin{align*}
\omega^{\frac{1}{2}a_j \left(\frac{m_j+1}{2}\right)^{b_j}}&=\omega^{\frac{1}{2K+2}\frac{2K+2}{2}a_j
\left(\frac{m_j+1}{2}\right)^{b_j}}=\widetilde{\omega}^{a_j \left(\frac{m_j+1}{2}\right)^{b_j} (K+1)}\\
&\le \widetilde{\omega}^{a_j^{1-\beta} (K+1)
\,(\log\,M)/\log\,\widetilde{\omega}^{-1}}=\left(\frac{1}{M^{K+1}}\right)^{a_j^{1-\beta}}.
\end{align*}

Without loss of generality, we assume $M\ge \ee$. Since $a_j\ge 1$
for $j\le j^*_{\beta,\delta,M}-1$ and $a_j\ge \exp(\delta j)$ for
$j\ge j^*_{\beta,\delta,M}$ we obtain
\begin{align*}
\gamma\le
A\,\left(\frac{j^*_{\beta,\delta,M}-1}{M^{K+1}}+\sum_{j=j^*_{\beta,\delta,M}}^\infty\left(\frac{1}{M^{K+1}}\right)^{\exp((1-\beta)
\delta j)}\right).
\end{align*}
Note that there exists a constant $C>0$ such that
\begin{align*}
j^*_{\beta,\delta,M}\leq (\log\log M)\,j^*_{\beta,\delta}
\end{align*}
with 
\begin{align*}
j^*_{\beta,\delta}:=C \max\left(j^*_{\delta}, \frac{1-\log\log \tilde{\omega}^{-1}}{\delta\beta}\right).
\end{align*}
Thus
\begin{align*}
\gamma\le
A\,\left(\frac{(\log\log M)j^*_{\beta,\delta}-1}{M^{K+1}}+\sum_{j=0}^\infty\left(\frac{1}{M^{K+1}}\right)^{\exp((1-\beta)
\delta j)}\right)\le\frac{C_{\beta,\delta}}{M^{K}},
\end{align*}
with
\begin{align*}
C_{\beta,\delta}:=A\,\left(j^*_{\beta,\delta}-1+\sum_{j=0}^\infty\left(\frac{1}{\ee}\right)^{\exp((1-\beta)
\delta j)-1}\right)<\infty,
\end{align*}
where we made use of $M\ge \ee$. Note that for $M\ge
C_{\beta,\delta}^{1/K}$ we have $\gamma\le 1$.

Using convexity we easily check that $-1+\exp(\gamma) \le
(\mathrm{e}-1)\gamma$ for all $\gamma\in[0,1]$. Thus for $M\ge
C_{\beta,\delta}^{1/K}$ we obtain
\begin{align*}
-1 + \prod_{j=1}^s \left(1 +\omega^{\frac{1}{2}a_j
\left(\frac{m_j+1}{2}\right)^{b_j}} A\right) \le
-1+\exp(\gamma)\le (\mathrm{e}-1)\gamma \le
\frac{C_{\beta,\delta}\,(\mathrm{e}-1)}{M^{K}}.
\end{align*}

We now turn to $|\mathcal{A}(s,M)|$. From the proof of
\cite[Theorem 9]{IKPW15} we get that
\begin{align*}
|\mathcal{A}(s,M)|\le2^{j^*_{\beta,\delta,M}}
\left(1+\frac{\log\,M}{\log\,\omega^{-1}}\right)^{B+(\log
2)/\delta}.
\end{align*}
as well as
\begin{align*}
\prod_{j=1}^{\min (s,j(x_M))} \left\lceil \left(\frac{\log M}{a_j
\log\omega^{-1}}\right)^{1/b_j}\right\rceil\le
2^{j^*_{\beta,\delta,M}}
\left(1+\frac{\log\,M}{\log\,\omega^{-1}}\right)^{B+(\log
2)/\delta}.
\end{align*}

Therefore
\begin{align*}
e^2_{n,s}\le
\frac{1}{M}\,\left[1+C_{\beta,\delta}(\mathrm{e}-1)4^{j^*_{\beta,\delta,M}}
\left(1+\frac{\log\,M}{\log\,\omega^{-1}}\right)^{2B+(2\log
2)/\delta}\right]\le \frac{D_{\beta,\delta}}{M^{1/2}},
\end{align*}
where
\begin{align*}
D_{\beta,\delta}:=\sup_{x\ge
C_{\beta,\delta}}\left(\frac1{x^{1/2}}+\frac{C_{\beta,\delta}(\mathrm{e}-1)(\log\,x)^{j_{\beta,\delta}^*\log4}}{x^{1/2}}\,
\left(1+\frac{\log\,x}{\log\,\omega^{-1}}\right)^{B+(\log
2)/\delta}\right)<\infty.
\end{align*}
Hence for
$M=\max(C_{\beta,\delta}^{1/K},D_{\beta,\delta}^{2}\,\e^{-4},\ee)$
we have
\begin{align*}
e_{n,s}\le \e.
\end{align*}
We estimate the number $n$ of function values used by the
algorithm $A_{n,s,M}$. We have
\begin{align*}
n&=\prod_{j=1}^s m_j=\prod_{j=1}^{\min(s,j^*_{\beta,\delta,M})}m_j\le\prod_{j=1}^{\min(s,j^*_{\beta,\delta,M})}
\left(1+2\left(\frac{\log\,M}{a_j^{\beta}\,\log\,\widetilde{\omega}^{-1}}\right)^{1/b_j}\right)\\
&\le 3^{j^*_{\beta,\delta,M}}\,\left(\frac{\log\,M}{\log\,\widetilde{\omega}^{-1}}\right)^B
=\mathcal{O}((1+\log\,\e^{-1})^{B+(\log 3)/(\beta\,\delta)}),
\end{align*}
where the factor in the big $\mathcal{O}$ notation depends only on
$\beta$ and $\delta$. This proves EC-SPT with
\begin{align*}
\tau=B+\frac{\log 3}{\beta\,\delta}.
\end{align*}
Since $\beta$ can be arbitrarily close to one, and $\delta$ can be
arbitrarily close to $\alpha^*$, the exponent $\tau^*$ of EC-SPT
is at most
\begin{align*}
B+\frac{\log 3}{\alpha^*},
\end{align*}
where for $\alpha^*=\infty$ we have $\frac{\log 3}{\alpha^*} = 0$.
This completes the proof of Theorem \ref{th:result}. $\hfill \qed$

\section{Relations to multivariate integration}\label{secint}

Multivariate integration
\begin{align*}
{\rm INT}_s(f)=\int_{\RR^s}f(\bsx) \varphi_s(\bsx)\,{\rm d}\bsx
\end{align*}
for $f$ from the Hermite space $\cH(K_{s,\bsa,\bsb,\omega})$ was
studied in \cite{IKLP14}. It is easy to see that multivariate
approximation using information from $\Lambda^{\rm std}$ is not easier than multivariate integration, see
e.g., \cite{NSW04}. More precisely, for any algorithm
$A_{n,s}(f)=\sum_{k=1}^n \alpha_k f(\bsx_k)$ for multivariate
approximation using the nodes $\bsx_1,\ldots,\bsx_n\in [0,1)^s$
and $\alpha_k\in \LL_2(\RR^s,\varphi_s)$, define
$\beta_k:=\int_{\RR^s} \alpha_k (\bsx) \varphi_s(\bsx) \rd \bsx$
and the algorithm
\begin{align*}
Q_{n,s}(f)=\sum_{k=1}^n\beta_k\,f(\bsx_k)
\end{align*}
for multivariate integration. Then
\begin{align*}
\abs{{\rm INT}_s(f) - Q_{n,s}(f)}&=\left|\int_{\RR^s}\left(f(\bsx)-\sum_{k=1}^n\alpha_k(\bsx)\,f(\bsx_k)\right)\varphi_s(\bsx) \,{\rm d}\bsx\right|\\
&\le\left(\int_{\RR^s}\left(f(\bsx)-\sum_{k=1}^n\alpha_k(\bsx)\,f(\bsx_k)\right)^2\varphi_s(\bsx)\,{\rm d}\bsx\right)^{1/2}\\
&=\|f-A_{n,s}(f)\|_{\LL_2}.
\end{align*}
This proves that for the worst-case error of integration we have
\begin{align*}
e^{{\rm
int}}(\cH(K_{s,\bsa,\bsb,\omega}),Q_{n,s}):=\sup_{\substack{f\in
\cH(K_{s,\bsa,\bsb,\omega})\\ \norm{f}_{K_{s,\bsa,\bsb,\omega}}\le
1}}\abs{{\rm INT}_s(f) - Q_{n,s}(f)}\le e^{{\rm
app}}(\cH(K_{s,\bsa,\bsb,\omega}),A_{n,s}).
\end{align*}
Since this holds for all linear approximation algorithms $A_{n,s}$
we conclude that
\begin{align}\label{eqintapp}
e(n,{\rm INT}_s):=\inf_{Q_{n,s}}e^{{\rm
int}}(\cH(K_{s,\bsa,\bsb,\omega}),Q_{n,s})\le e(n,\APP_s;\lstd),
\end{align}
where $e(n,{\rm INT}_s)$ denotes the $n$th minimal (worst-case)
error of integration. 

Furthermore for $n=0$ we have equality,
\begin{align*}
e(0,{\rm INT}_s)=e(0,\APP_s)=1.
\end{align*}
From these observations it follows that for $\varepsilon \in
(0,1)$ and $s \in \NN$ we have
\begin{align}\label{eqintapp2}
n(\varepsilon,{\rm INT}_s) \le n(\varepsilon,\APP_s;\lstd),
\end{align}
where $n(\varepsilon,{\rm INT}_s)$ is the information complexity
for the integration problem.

The inequalities \eqref{eqintapp} and \eqref{eqintapp2} mean that
all positive results for multivariate approximation also hold for
multivariate integration.

In \cite{IKLP14} the following results were proved:
\begin{thm}[{\cite[Theorem~1]{IKLP14}}]\label{thm2_int}
For the integration problem over the Hermite space
$\cH(K_{s,\bsa,\bsb,\omega})$ we have:
\begin{enumerate}
\item {\rm EXP} holds for all $\bsa$ and $\bsb$ considered, and
\begin{align*}
p^{*}(s)=\frac{1}{B(s)} \ \ \ \ \ \mbox{with}\ \ \ \ \
B(s):=\sum_{j=1}^s\frac1{b_j}.
\end{align*}
\item The following assertions are equivalent:
\begin{enumerate}
\item The $b_j^{-1}$'s are summable, i.e., $B:=\sum_j
b_j^{-1}<\infty$; \item we have {\rm UEXP}; \item we have {\rm
EC-PT}; \item we have {\rm EC-SPT}.
\end{enumerate}
If one of the assertions holds then $p^*=1/B$ and the exponent
$\tau^{\ast}$ of {\rm EC-SPT} is $B$. \item\label{wtnec} {\rm
EC-WT} implies that $\lim_{j \rightarrow \infty} a_j 2^{b_j}
=\infty$. \item\label{wtsuff} A sufficient condition for {\rm
EC-WT} is that there exist $\eta>0$ and $\beta>0$ such that $$a_j
2^{b_j} \ge \beta j^{1+\eta}\ \ \ \ \mbox{ for all }\ \ j \in
\NN.$$
\end{enumerate}
\end{thm}

Compared with our results for approximation from Theorem
\ref{th:result} we have:
\begin{itemize}
\item The conditions for EXP and for UEXP are the same for both
problems. \item For the integration problem UEXP and EC-SPT are
equivalent and these properties only depend on $\bsb$ but not on
$\bsa$. This makes a difference to the approximation problem where
we have the same condition on $\bsb$ as for the integration
problem in order to achieve UEXP. However, to obtain also EC-SPT for
approximation we must require that the sequence $\bsa$ grows at an
exponential rate. 
\item  For the integration problem there is a
gap between the necessary and sufficient condition for EC-WT
whereas for the approximation problem we have an {\it if and only
if} condition. With the help of our result for EC-WT for
approximation and with our previous considerations we can present a 
different sufficient condition for EC-WT for integration as compared to
\cite[Theorem~1]{IKLP14} (see Point 4 of Theorem~\ref{thm2_int}).
\begin{thm}\label{suffECWT}
A sufficient condition for {\rm EC-WT} for integration in
$\cH(K_{s,\bsa,\bsb,\omega})$ is that $\lim_{j \rightarrow \infty}
a_j =\infty$.
\end{thm}
\begin{proof}
Assume that we have $\lim_{j \rightarrow \infty} a_j =\infty$. Then
Theorem \ref{th:result} implies that we have EC-WT for the
approximation problem. But now it follows easily from
\eqref{eqintapp2} that we also have EC-WT for the integration
problem.
\end{proof}

Although Theorem \ref{suffECWT} is in some cases an improvement of the
sufficient condition for EC-WT for integration from \cite[Theorem~1]{IKLP14} there still remains
a small gap to the necessary condition.

\end{itemize}

\begin{small}
\noindent\textbf{Authors' addresses:}
\\ \\
\noindent Christian Irrgeher, Peter Kritzer, Friedrich Pillichshammer,
\\ 
Department of Financial Mathematics and Applied Number Theory, 
Johannes Kepler University Linz, Altenbergerstr.~69, 4040 Linz, Austria\\
 \\
\noindent Henryk Wo\'{z}niakowski, \\
Department of Computer Science, Columbia University, New York 10027,
USA, and Institute of Applied Mathematics, 
University of Warsaw, ul. Banacha 2, 02-097 Warszawa, Poland\\ \\

\noindent \textbf{E-mail:} \\
\texttt{christian.irrgeher@jku.at}\\
\texttt{peter.kritzer@jku.at}\\
\texttt{friedrich.pillichshammer@jku.at} \\
\texttt{henryk@cs.columbia.edu}
\end{small}

\end{document}